\documentclass[12pt]{amsart}

\usepackage{amsfonts,latexsym,rawfonts,amsmath,amssymb,amsthm,mathrsfs}
\usepackage[plainpages=false]{hyperref}
\usepackage{graphicx}
\usepackage{comment}

\numberwithin{equation}{section}

\RequirePackage{color}
\theoremstyle{plain}

\newtheorem{definition}{Definition}[section]

\newtheorem{lemma}{Lemma}[section]

\newtheorem{theorem}{Theorem}[section]
\newcommand{\beq}{\begin{equation}}
\newcommand{\eeq}{\end{equation}}
\newcommand{\beqs}{\begin{eqnarray*}}
\newcommand{\eeqs}{\end{eqnarray*}}
\newcommand{\beqn}{\begin{eqnarray}}
\newcommand{\eeqn}{\end{eqnarray}}
\newcommand{\beqa}{\begin{array}}
\newcommand{\eeqa}{\end{array}}

\def\S{\mathbb S}
\def\R{\mathbb R}
\def\tR{\mathcal R}

\def\A{\mathscr A}
\def\tC{\mathcal C}
\def\M{\mathcal M}
\def\N{\mathcal N}

\def\J{\mathcal J}
\def\I{\mathcal I}

\def\O{\mathcal O}

\def\p{\partial}
\def\vol{\text{Vol}}

\def\const{\text{const.}}
\def\Jac{\text{Jac}}

\def\eps{\epsilon}

\def\h{h_*}
\def\b{b^*}
\def\u{u^*}
\def\r{r^*}

\begin{document}


\title{Flow by Gauss curvature to\\ the Aleksandrov and dual Minkowski problems}

\author{Qi-Rui Li}
\address{Qi-Rui Li: 
Centre for Mathematics and Its Applications, the Australian National University, Canberra, ACT 2601, Australia.}
\email{qi-rui.li@anu.edu.au}

\author{Weimin Sheng}
\address{Weimin Sheng: School of Mathematical Sciences, Zhejiang University, Hangzhou 310027, China; and
Centre for Mathematics and Its Applications, the Australian National University, Canberra, ACT 2601, Australia.}
\email{weimins@zju.edu.cn}

\author{Xu-Jia Wang}
\address{Xu-Jia Wang: 
Centre for Mathematics and Its Applications, the Australian National University, Canberra, ACT 2601, Australia.}
\email{xu-jia.wang@anu.edu.au}



\thanks{The first and third authors were supported by ARC FL130100118.
The second author was supported by NSFC: 11571304, 11131007.}

\begin{abstract}
In this paper we study a contracting flow of closed, 
convex hypersurfaces in the Euclidean space $\R^{n+1}$
with speed $f r^{\alpha} K$, where $K$ is the Gauss curvature, 
$r$ is the distance from the hypersurface to the origin,
and $f$ is a positive and smooth function.
If $\alpha \ge n+1$, we prove that the flow exists for all time
and converges smoothly after normalisation to a soliton,
which is a sphere centred at the origin if $f \equiv 1$.
Our argument provides a parabolic proof in the smooth category for the classical Aleksandrov problem, 
and resolves the dual q-Minkowski problem introduced by Huang, Lutwak, Yang and Zhang \cite{HLYZ16}, 
for the case $q<0$.
If $\alpha < n+1$, corresponding to the case $q>0$, 
we also establish the same results for even function $f$ and origin-symmetric initial condition,
but for non-symmetric $f$, counterexample is given for the above smooth convergence.
\end{abstract}

\maketitle

\baselineskip16pt
\parskip3pt


\section{Introduction}

Flow generated by the Gauss curvature was first studied by Firey \cite{Fir74} to
model the shape change of tumbling stones.
Since then the evolution of hypersurfaces   by their Gauss curvature has been studied by many authors
 \cite{And96}-\cite{AndGuNi16}, \cite{BCD16}-\cite{ChWang00}, \cite{DaskLee04, GuNi16, Ha94}.
A main interest is to understand the asymptotic behavior of the flows.
It was conjectured that the $\alpha$-power of the Gauss curvature, for $\alpha>\frac{1}{n+2}$, 
deforms a convex hypersurface in $\R^{n+1}$ into a round point.
This is a difficult problem and has been studied by many authors in the last three decades.
The first result was by Chow \cite {Chow85} who provided a proof for the case $\alpha=1/n$.
In \cite{And99} Andrews proved the conjecture for the case $n=2$ and $\alpha=1$. 
Very recently, Brendle, Choi and Daskalopoulos \cite{BCD16}
resolved the conjecture for all $\alpha>\frac{1}{n+2}$, in all dimensions.

As a natural extension, anisotropic flows have also attracted much attention
and have been extensively investigated  \cite{ChZhu99,Ga93,GaLi94}.
They provide alternative proofs for the existence of solutions to elliptic PDEs arising in geometry and physics.
For example a proof based on the logarithmic Gauss curvature flow was given in \cite{ChWang00} 
for the classical Minkowski problem, and in \cite {Wang96} for a prescribing Gauss curvature  problem.
Expansion of convex hypersurfaces by their Gauss curvature has also been studied by several authors 
 \cite{Gerh90,Gerh14, Li10, Sch06, Urb91}.

Let $\M_0$ be a smooth, closed, uniformly convex hypersurface in $\R^n$ enclosing the origin.
In this paper we study the following anisotropic Gauss curvature flow,
\beq\label{flow}
\left\{ { \begin{split}  
  \frac{\p X}{\p t} (x,t) &= - f(\nu) r^{\alpha} K(x,t) \nu,\\
  X(x,0) &=X_0(x),
 \end{split} }\right.
\eeq
where $K(\cdot,t)$ is the Gauss curvature of hypersurface $\M_t$, 
parametrized by $X(\cdot,t):\S^n\to \R^{n+1}$, 
$\nu(\cdot,t)$ is the unit outer normal at $X(\cdot, t)$,
and $f$ is a given positive smooth function on $\S^n$.
We denote  by $r = |X(x,t)|$ the distance from the point $X(x, t)$ to the origin,
and regard it as a function of $\xi=\xi (x,t) := X(x,t) / |X(x,t)|\in \S^n $.
We call it the radial function of $\M_t$.

When $\alpha\ge n+1$, we prove that if $f\equiv 1$,
the hypersurface $\M_t$ converges smoothly after normalisation to a sphere. 
For general positive and smooth function $f$, we prove that 
$\M_t$ converges smoothly after normalisation to a hypersurface 
which is a solution to the classical Aleksandrov problem \cite {Aleks42} ($\alpha=n+1$) 
and to the dual $q$-Minkowski problem  \cite {HLYZ16}  for $q\le 0$ ($\alpha>n+1$).
Our proof of the smooth convergence consists of two parts:
\begin{itemize}
\item [(i)] uniform positive upper and lower bounds for  the radial function of  $\widetilde \M_t$; and 
\item [(ii)] uniform positive upper and lower bounds for the principal curvatures of  $\widetilde \M_t$,
\end{itemize}
where $\widetilde \M_t$ is the normalised solution given in \eqref{rescaled surface} below. 
Once the upper and lower bounds for the principal curvatures are established, 
higher order regularity of $\widetilde \M_t$ follows from Krylov's regularity theory. 
We then infer the smooth convergence by using the functional \eqref {functional}.
Our proof of part (ii) applies to the flow \eqref{flow} for all $\alpha\in\R^1$, as long as part (i) is true.
In particular it also applies to the original Gauss curvature flow (namely the case $\alpha=0$) 
for which the estimates (ii) were established for $f\equiv 1$ in \cite{GuNi16}.

When $\alpha<n+1$,  we establish the smooth convergence for even $f$, 
provided the initial hypersurface is symmetric with respect to the origin.
We also give examples to show that, without the symmetry assumption,
part (i) above fails and so the smooth convergence does not hold.

As a result we also obtain the existence of smooth symmetric solutions 
to the dual $q$-Minkowski problem for all $q\in\R^1$, 
assuming the function $f$ is smooth, positive, and $f$ is even when $q>0$.
The dual $q$-Minkowski problem was recently introduced
by Huang, Lutwak, Yang, and Zhang \cite {HLYZ16}
where they proved the existence of symmetric weak solutions for the case $q \in (0, n+1)$
under some conditions.
Their conditions were recently improved by Zhao \cite{ZY17}.
For $q<0$ the existence and uniqueness of weak solution were obtained in \cite {ZY}.
When $q=n+1$ it is the logarithm Minkowski problem studied in \cite {BLYZ}. 
In \cite {BLYZ} and  \cite {HLYZ16}, the existence of weak solutions was proved when the inhomogeneous term is
a non-negative measure not concentrated in any sub-spaces. 
For other related results, we refer the readers to \cite{BHP17,BLYZZ,Schn14}
and the references therein.
 
Let us state our first main result as follows.
 
\begin{theorem}\label{thmA}
Let $\M_0$ be a smooth, closed, uniformly convex hypersurface in $\R^{n+1}$ enclosing the origin.
If $f \equiv 1$ and $\alpha \ge n+1$,
then the flow \eqref{flow} has a unique smooth solution $\M_t$ for all time $t>0$,
which converges to the origin.
After a proper rescaling $X\to \phi^{-1}(t)X$, 
the hypersurface $\widetilde \M_t = \phi^{-1}(t) \M_t$ converges exponentially fast
to the unit sphere centred at the origin in the $C^\infty$ topology.
\end{theorem}

Our choice for the rescaling factor $\phi(t)$ is motivated by the following calculation.
Assume
\beq\label{homothety}
X(\cdot,t) = \phi(t) X_0(\cdot)
\eeq
evolves under the flow \eqref{flow} with initial data $\phi_0X_0$,
where $\phi$ is a positive function and $\phi_0 = \phi(0)$.
Since the normal vector is unchanged by the homothety,
we obtain, 
by differentiating \eqref{homothety} in $t$ and multiplying $\nu_0=\nu(\cdot,t)$ to both sides,
\beq\label{s1 t1}
\phi'(t)\langle X_0, \nu_0 \rangle = -\phi^{\alpha-n} (t) f r_0^\alpha K_0,
\eeq
where $K_0$ is the Gauss curvature of $\M_0 = X_0(\S^n)$,
and $r_0$  is the radial function of $\M_0$.
By \eqref{s1 t1} we have
\beqs
\phi'(t) = - \lambda \phi^{\alpha - n} (t)
\eeqs
for some constant $\lambda>0$.
We may suppose $\lambda =1$. Then
\beq\label{scaling factor}
{ \begin{split}  
  \phi (t) &= \phi_0 e^{-t},\;\;&\text{if}\;\; \alpha=n+1,\\
  \phi (t) &= [\phi_0^{q} -q t]^{\frac{1}{q}},\;\;&\text{if}\;\; \alpha \not=n+1 , 
 \end{split} }
\eeq
where $q =n+1- \alpha $, $\phi_0 = \phi(0)>0$.
By \eqref{s1 t1}, one sees that $\M_0$ satisfies the following elliptic equation
\beq\label{soliton sol}
\frac{u(x)}{r^\alpha(\xi) K(p)} = f(x) \ \ \ \forall\ x\in \S^n,
\eeq
where $p \in \M_0$ is the point such that the unit outer normal $\nu(p)=x$, $\xi = p/|p| \in \S^n$,
and $u$ is the support function of $\M_0$, given by
\beqs
u(x)  = \sup\{\langle x, y \rangle:\ y\in \M_0\}.
\eeqs
The above calculation suggests that if we expect that our flow converges to a soliton
which satisfies \eqref{soliton sol}, it is reasonable to rescale the flow by a time-dependent factor
$\phi(t)$ which is in the form of \eqref{scaling factor}.

Let us introduce the normalised flow for \eqref{flow}.
Let
\beq\label{rescaled surface}
{\begin{split}
\widetilde M_t & = \phi^{-1}(t) \M_t, \\
\widetilde X(\cdot, \tau) & = \phi^{-1}(t) X(\cdot,t) ,
\end{split}} \eeq
where
$$
\tau = \left\{
{\begin{split}
 & t \hskip80pt \text{if}\ \alpha = n+1,\\
 & \frac{1}{q} \log\frac{\phi_0^{q}}{\phi_0^{q}-qt} \ \ \ \ \ \text{if}\ \alpha \ne n+1.
 \end{split}} \right.
$$
Then $\widetilde X(\cdot,\tau)$ satisfies the following normalised flow
\beq\label{normalised flow}
\left\{{\begin{split}
\frac{\p X}{\p t} (x,t) &= - f(\nu) r^\alpha K(x,t) \nu + X(x,t),\\
X(\cdot,0) &= \phi_0^{-1} X_0.
\end{split}}\right.
\eeq 
For convenience we still use $t$ instead of $\tau$ to denote the time variable and omit the ``tilde"
if no confusions arise. 

The asymptotic behavior of \eqref{flow}
is equivalent to  the long time behaviour of the normalised flow \eqref{normalised flow}.
Indeed, in order to prove Theorem \ref{thmA},
we shall establish the a priori estimates for \eqref{normalised flow},
and show that $|X| \to 1$ smoothly as $t\to \infty$, provided $f\equiv 1$ and $\phi_0$ is chosen such that
\beq\label{phi_0}
{ \begin{split}  
  \phi_0 = \exp(\frac{1}{o_n}\int_{\S^n} \log r_0(\xi)d\xi), &\;\;\text{if}\;\alpha=n+1,\\
  \min_{\S^n}\, r_0(\cdot)  \le \phi_0 \le \max_{\S^n}\, r_0(\cdot), &\;\;\text{if}\; \alpha>n+1,
 \end{split} }
\eeq
where $o_n = |\S^n|$ denotes the area of the sphere $\S^n$.

The following functional plays an important role in our argument,
\beq\label{functional}
\J_\alpha (\M_t) =
\left\{{ \begin{split}  
  \int_{\S^n} f(x) \log u(x,t) dx- \int_{\S^n} \log r(\xi,t)d\xi, &\;\;\text{if}\;\alpha=n+1,\\
  \int_{\S^n} f(x) \log u(x,t) dx - \frac1q\int_{\S^n} r^q(\xi,t)d\xi, &\;\;\text{if}\; \alpha \not = n+1,
 \end{split} }\right.
\eeq
where $q = n+1-\alpha $ as above,
$u(\cdot, t)$ and $r(\cdot, t)$ are respectively the support function and radial function of $\M_t$.
This functional was introduced in \cite{HLYZ16}.
We will show in Lemma \ref{descending flow} below that
$\J_\alpha(\M_t)$ is strictly decreasing unless $\M_t$ solves the elliptic equation \eqref{soliton sol}.
By this functional and the a priori estimates for the normalised flow \eqref{normalised flow},
we obtain the following convergence result for the anisotropic flow \eqref{flow}.

\begin{theorem}\label{thmB}
Let $\M_0$ be a smooth, closed, uniformly convex hypersurface in $\R^{n+1}$
which contains the origin in its interior. 
Let $f$ be a smooth positive function on $\S^n$.
If $\alpha>n+1$, then the flow \eqref{flow} has a unique smooth solution $\M_t$ for all time $t>0$.
When $t\to \infty$, the rescaled hypersurfaces $\widetilde \M_t$
converge smoothly to the unique smooth solution of  \eqref{soliton sol},
which is a minimiser of the functional \eqref{functional}.
\end{theorem}

When $\alpha = n+1$, in order that the solution of \eqref{flow} converges to a solution of \eqref{soliton sol},
we assume that $f\in C^\infty(\S^n;\R_+)$ and satisfies the following conditions
\beqn
 &&\int_{\S^n} f = o_n:=|\S^n|, \label{Aleks f cdt1}\\
 && \int_{\omega} f <|\S^n|- |\omega^*|  \label{Aleks f cdt2}
\eeqn
for any spherically convex subset $\omega\subset \S^n$.
Here $|\cdot|$ denotes the $n$-dimensional Hausdorff measure,
and $\omega^* \subset \S^n$ is the dual set of $\omega$,
namely $\omega^* = \{\xi\in \S^n:\;\;x\cdot \xi \le 0,\;\;\forall\; x\in \omega\}$.

\begin{theorem}\label{thmC}
Let $\M_0$ be  as in Theorem \ref{thmB}.
Assume $\alpha=n+1$ and \eqref{Aleks f cdt1}, \eqref{Aleks f cdt2} hold.
Then \eqref{flow} has a unique smooth solution $\M_t$ for all time $t>0$.
When $t\to \infty$, the rescaled hypersurfaces $\widetilde \M_t$
converge smoothly  to the smooth solution of \eqref{soliton sol}, 
which is a minimiser of the functional \eqref{functional}.
\end{theorem}

Theorem \ref{thmC} gives a proof for the classical Aleksandrov problem
in smooth category by a curvature flow approach.
We point out that conditions \eqref{Aleks f cdt1} and \eqref{Aleks f cdt2}
are necessary for Aleksandrov's problem \cite{Aleks42}, 
but for the flow \eqref{flow}, condition  \eqref{Aleks f cdt1} is satisfied 
by any bounded positive function $f$ provided we make a scaling of the time $t$. 
At the end of the paper we will show that Theorem \ref{thmC} does not hold 
if \eqref{Aleks f cdt2} is violated.

Let $\M$ be a convex hypersurface in $\R^{n+1}$ with the origin $\O$ in its interior.
Then $\M$ is a spherical radial graph via the mapping
\beqs
\vec{\hskip1pt r}: \xi\in \S^n \mapsto r(\xi)\xi \in \M.
\eeqs
Let $\A=\A_{\M}$ be a set-valued mapping given by
\beqs
\A(\omega) = \cup_{\xi\in \omega}\{\nu(\vec{\hskip1pt r}(\xi)) \},
\eeqs
where $\nu$ is the Gauss map of $\M$.
Aleksandrov raised the following problem:
given a finite nonnegative Borel measure $\mu$ on $\S^n$,
whether there exists a convex hypersurface $\M$ such that
\beq\label{Aleks problem}
|\A(\omega) |= \mu(\omega)\ \ \forall\ \text{Borel sets}\  \omega\subset \S^n.
\eeq
The left hand side of \eqref{Aleks problem} defines a measure on $\S^n$,
which is called the integral Gauss curvature of $\M$.
The existence and uniqueness (up to a constant multiplication) of weak solution to this problem
were obtained by Aleksandrov \cite{Aleks42},
assuming that $\mu$ is nonnegative, $\mu(\S^n) = o_n$
and $\mu (\S^n \setminus \omega) > |\omega^*|$ for any convex $\omega \subset \S^n$.
These conditions are equivalent to \eqref{Aleks f cdt1} and \eqref{Aleks f cdt2},
if $\mu$ has a density function $f$.
If $\M$ is a hypersurface with prescribed integral Gauss curvature $\mu$,
then its polar dual
\beq\label{polar dual}
\M^*=\p \{z\in \R^{n+1}:\;z\cdot y \le 1\;\;\forall \ y\in \M\}\
\eeq
solves \eqref{soliton sol} for $\alpha=n+1$.

For general $\alpha$, the limiting hypersurface of the flow \eqref{flow} is related to
the dual Minkowski problem introduced most recently in \cite{HLYZ16}.
Given a real number $q$ and a finite Borel measure $\mu$ on the sphere $\S^n$,
the authors asked if there exists a convex body $\Omega$ with the origin inside
such that its $q$-th dual curvature measure 
\beq\label{dual Mink prob}
\widetilde C_q(\Omega,\cdot) = \mu(\cdot).
\eeq
Denote by $\M$ the boundary of $\Omega$, and by $\A^* = \A^*_\M$ the ``inverse" of $\A_\M$,
namely
\beqs
\A^* (\omega) =\{\xi \in \S^n\ : \ \nu(\vec{\hskip1pt r}(\xi))\in\omega\} . 
\eeqs
The $q$-th dual curvature measure is defined by
\beq\label{dual curvature meas}
\widetilde C_q(\Omega,\omega) = \int_{\A^*(\omega)} r^{q}(\xi)d\xi.
\eeq
Hence the dual Minkowski problem \eqref{dual Mink prob} is equivalent to the equation
\beq\label{s2 a1}
r^q  |\Jac \A^* | = f \ \ \text{on} \ \S^n,
\eeq
provided $\mu$ has a density function $f$.
Here $|\Jac \A^*|$ denotes the determinant of the Jacobian of the mapping
$x \mapsto \xi = \A_\M^*(x)$.
By \eqref{s2 t9} below, we see that 
the dual Minkowski problem is equivalent to the solvability of the equation \eqref{soliton sol} with $\alpha=n+1-q$.
Noting that
\beq {\begin{split}
\A_{\M}^*(\omega)
&= \{\xi \in \S^n\ : \ \nu(\vec{\hskip1pt r}(\xi))\in\omega\} \\
&= \{\nu^*(\vec{\hskip1pt r}^*(x)): \  x\in \omega\}  \\
&= \A_{\M^*} (\omega),                          
\end{split}} \eeq
where $\nu$ and $\nu^*$ denote the unit outer normal of $\M$ and $\M^*$ respectively,
we also see that if $\M^*$ solves the Aleksandrov problem \eqref{Aleks problem},
then $\M$ solves the dual Minkowski problem \eqref{dual Mink prob} for $q=0$, 
and so is a solution to \eqref{soliton sol} for $\alpha=n+1$. 

When $\alpha<n+1$, we consider the behaviour of origin-symmetric hypersurfaces under the flow \eqref{flow},
assuming that $f$ is an even function, namely $f(x) = f(-x)$ for all $x\in \S^n$.
In this case the solution $\M_t$ shrinks to a point in finite time, namely  as $t\to T$ for some $T<\infty$.
Our next theorem shows that the normalised solution converges smoothly if $f$ is smooth and positive.

\begin{theorem}\label{thmDa}
Let $\M_0$ be a smooth, closed, uniformly convex, and origin-symmetric hypersurface in $\R^{n+1}$.
Let $ \alpha<n+1$.
If $f$ is a smooth, positive, even function on $\S^n$,
then the flow \eqref{flow} has a unique smooth solution $\M_t$.
After normalisation, the rescaled hypersurfaces $\widetilde \M_t$ converge smoothly to
a smooth solution of \eqref{soliton sol},
which is a minimiser of the functional \eqref{functional}.
Moreover, if $f\equiv 1$ and $0\le \alpha<n+1$, 
then $\widetilde \M_t$ converge smoothly to a sphere.
\end{theorem}

In the proof of Theorem \ref{thmDa}, 
we will choose the constant $\phi_0$ in the rescaling \eqref{scaling factor} by
\beq\label{thmDa phi0}
\phi_0=\Big(\int_{\S^n} r_0^q(\xi)d\xi \Big \slash \int_{\S^n} f(x) dx \Big)^{\frac1q},
\eeq
where $r_0$ is the radial function of the  initial convex hypersurface $\M_0$.
This choice is such that the functional $\I_q$ in  \eqref{sym Lq-r} is a constant. 
This property is crucial for the uniform positive upper and lower bounds
for the support function in the normalised flow \eqref{normalised flow}.

Without the symmetry assumption, Theorem \ref{thmDa} is not true.
In fact, when $\alpha <n+1$, we find that the hypersurfaces evolving by \eqref{flow} may
reach the origin in finite time, before the hypersurface shrinks to a point.
Therefore the smooth convergence 
does not hold in general.

\begin{theorem} \label{thmD}
Suppose $n\ge 1$ and $\alpha < n+1$.
There exists a smooth, closed, uniformly convex hypersurface $\M_0$,
such that under the flow \eqref{flow},  
\beq\label{unbounded ratio}
 \tR (X(\cdot,t)) := \frac{\max_{\S^n} r(\cdot,t)}{\min_{\S^n} r(\cdot,t)} \to \infty\;\;\text{as}\;\;t\to T
\eeq
for some $T>0$.
\end{theorem}
  
Equation \eqref{soliton sol} can be written, in terms of the support function $u$,
as a Monge-Amp\`ere equation on the sphere, 
\beq\label{soliton sol-u}
\det(\nabla^2 u+uI)=\frac {f(x)}{u(x)} (|\nabla u|^2+u^2)^{\alpha/2}\ \ \text{on}\ \S^n.
\eeq
By Theorems \ref{thmA}-\ref{thmDa}, we have the following existence results for equation \eqref{soliton sol-u}.

\begin{theorem} \label{thmE}
Let $f$ be a smooth and positive function on the sphere $\S^n$. 
\newline
(i) If $\alpha>n+1$, there is a unique smooth, uniformly convex solution to \eqref {soliton sol-u}.
\newline
(ii) If $\alpha=n+1$ and $f$ satisfies \eqref {Aleks f cdt1}, \eqref{Aleks f cdt2},
there is a smooth, uniformly convex solution to \eqref {soliton sol-u}. The solution is unique up to dilation.
\newline
(iii) If $\alpha<n+1$ and $f$ is even,  there is an origin-symmetric solution to \eqref {soliton sol-u}.
\newline
(iv) If $f\equiv 1$, then the solution must be a sphere when $\alpha\ge n+1$,
and the origin-symmetric solution must be a sphere when $0\le \alpha<n+1$.
\end{theorem}

In case (ii) of Theorem \ref{thmE}, 
the existence and uniqueness (up to dilation) of the solution were proved by \cite{Aleks42},
and the regularity of the solution was obtained in \cite{Olik83,Pog73}.
In this paper we use the generalised solution to the Aleksandrov problem as a barrier
to establish the uniform estimate for the corresponding Gauss curvature flow. 
Our main concern of this paper is the smooth convergence of the flow, 
which also provides an alternative proof for the regularity of the solution.

It is interesting to compare equation \eqref{soliton sol-u} with the $L_p$-Minkowski problem
\beq\label{p-Minkow}
\det(\nabla^2 u+uI)=\frac {f(x)}{u^{1-p}(x)}  \ \ \text{on}\ \S^n.
\eeq
For equation \eqref{p-Minkow}, 
there is a solution if $p>-n-1$ \cite {ChWang06} and no solution in general if $p\le -n-1$.
In Theorem \ref{thmE} we proved the existence of solutions to \eqref{soliton sol-u} 
for all $\alpha\in\R^1$, which looks  stronger.
This is due to the associated functional \eqref{functional}, in which 
the first integral $\int_{\S^n} f\log u$ is bounded for our solution. 
This property, together with \eqref{sym Lq-r}, 
enables us to establish a uniform bound for the support  function $u$ (Lemma \ref{s3 lem1c}).

This paper is organised as follows.
In Section 2 we collect some properties of convex hypersurfaces, and show that
the flow \eqref{flow} can be reduced to a scalar parabolic equation of Monge-Amp\`ere type,
via the support function or the radial function.
We will also show in Section 2 that \eqref{normalised flow}
is a descending gradient flow of the functional \eqref{functional}.
In Section 3 we establish the uniform positive upper and lower bounds for  the support function 
of the normalised flow \eqref{normalised flow}.
The uniform positive upper and lower bounds for the principal curvatures
are proved in Section 4.
The a priori estimates ensure the longtime existence and the convergence of the normalised flow.
The proofs of Theorems \ref{thmA}-\ref{thmDa} will be presented in Section 5.
Finally in Section 6 we prove Theorem \ref{thmD}.

\vspace{2mm}
\noindent{\bf Acknowledgement} The authors would like to thank the referees
for careful reading of the manuscript and their helpful comments.
In particular, we would like to thank a referee who provided the proof for the uniqueness
in the case $f\equiv 1$ and $0\le \alpha<n+1$ in Theorem 1.4 and Theorem 1.6 (iv).
He/She also pointed out the monotonicity of the functional $\J_{n+1}$ in Lemma \ref{lem2.2}.

\section{Preliminaries}

Let us first recall  some basic properties of convex hypersurfaces.
Let $\M$ be a smooth, closed, uniformly convex hypersurface in $\R^{n+1}$.
Assume that $\M$ is parametrized by the inverse Gauss map $X:\;\S^n \to \M$.
The support function $u:\S^n \to \R$ of $\M$ is defined by
\beq\label{s2 t1}
 u(x) = \sup\{\langle x,y\rangle:\;y\in \M\}.
\eeq
The supremum is attained at a point $y$ such that $x$ is the outer normal of $\M$ at $y$.
It is easy to check that
\beq\label{s2 t2}
y = u(x) x+\nabla u(x),
\eeq
where $\nabla$ is the covariant derivative with respect to the standard metric $e_{ij}$ of the sphere $\S^n$.
Hence
\beq\label{s2 r}
 r = |y|= \sqrt{u^2 + |\nabla u|^2}.
\eeq
The second fundamental form of $\M$ is given by, see e.g. \cite{And00,Urb91},
\beq\label{s2 t3}
  h_{ij} =u_{ij} + ue_{ij},
\eeq
where $u_{ij} = \nabla^2_{ij} u$ denotes the second order covariant derivative of $u$
with respect the spherical metric $e_{ij}$.
By Weingarten's formula,
\beq\label{s2 t4}
e_{ij} =\big \langle\frac{\p \nu}{\p x_i}, \frac{\p \nu}{\p x_j}\big\rangle = h_{ik}g^{kl}h_{jl},
\eeq
where $g_{ij}$ is the metric of $\M$ and $g^{ij}$ its inverse.
It follows from \eqref{s2 t3} and \eqref{s2 t4} that
the principal radii of curvature of $\M$, under a smooth local orthonormal frame on $\S^n$,
are the eigenvalues of the matrix
\beq\label{principal radii}
  b_{ij} = u_{ij} + u\delta_{ij}.
\eeq
In particular the Gauss curvature is given by
\beq\label{s2 Gauss}
K = 1/\det(u_{ij} + u \delta_{ij}) = S_n^{-1}(u_{ij} + u \delta_{ij}),
\eeq
where
$$S_k= \sum_{i_1<\cdots<i_k} \lambda_{i_1}\cdots\lambda_{i_k}$$
denotes the $k$-th elementary symmetric polynomial.

Let $X(\cdot,t)$ be a smooth solution to the normalised flow \eqref{normalised flow}
and let $u(\cdot,t)$ be its support function.
From the above discussion we see that the flow \eqref{normalised flow} can be reduced to
the initial value problem for the support function $u$:
\beq\label{normalised flow spt}
{\left\{
\begin{split}
 \frac{\p u}{\p t}(x,t) &= - f(x) r^\alpha S_n^{-1}(u_{ij} + u \delta_{ij})(x,t) + u(x,t)
                                      \;\;\text{on}\;\S^n\times[0,\infty), \\
 u(\cdot,0) &= \widetilde u_0 := \phi^{-1}_0 u_0,
\end{split}\right.}
\eeq
where $r=\sqrt{u^2+|\nabla u|^2}(x,t)$ as in \eqref{s2 r},
$u_0$ is the support function of the initial hypersurface $\M_0$,
and $\phi_0$ is the dilation constant in  \eqref{normalised flow}.

As $\M$ encloses the origin, it can be parametrized via the radial function $r: \S^n \to \R_+$,
$$\M = \{r(\xi)\xi:\; \xi\in \S^n\}.$$
The following formulae are well-known, see e.g. \cite{Gerh14},
\beq\label{s2 t5}
\nu  = \frac{r \xi - \nabla r}{\sqrt{r^2+|\nabla r|^2}},
\eeq
\beq\label{s2 t6}
{\begin{split}
g_{ij} &= r^2 e_{ij} + r_ir_j,\\
h_{ij} &=  \frac{r^2 e_{ij} + 2r_ir_j - rr_{ij} }{\sqrt{r^2+|\nabla r|^2}}.
\end{split}}
\eeq
Set
\beq\label{def v}
v =  \frac{r}{u} =  \sqrt{1+|\nabla \log r|^2},
\eeq
where the last equality follows by multiplying $\xi$ to both sides of \eqref{s2 t5}.
The normalised flow \eqref{normalised flow} can be also described by
the following scalar equation for $r(\cdot,t)$,
\beq\label{normalised flow rad}
{\left\{
\begin{split}
 \frac{\p r}{\p t}(\xi,t) &= - v f r^\alpha K(\xi,t) +r(\xi,t) \;\;\text{on}\;\S^n\times[0,\infty), \\
 r(\cdot,0) &= \widetilde r_0 := \phi^{-1}_0 r_0,
\end{split}\right.}
\eeq
where $r_0$ is the radial function of  $\M_0$,
and $K(\xi,t)$ denotes the Gauss curvature at $r(\xi,t)\xi \in \M_t$.
Note that in \eqref{normalised flow rad} $f$ takes value at $\nu = \nu(\xi,t)$ given by \eqref{s2 t5}.
By \eqref{s2 t6} we have, under a local orthonormal frame on $\S^n$,
\beq\label{s2 t7}
K =  \frac{\det h_{ij}}{\det g_{ij}} = v^{-n-2} r^{-3n} \det (r^2\delta_{ij} + 2r_ir_j-rr_{ij}).
\eeq

Given any $\omega \subset \S^n$, let $\tC = \tC_{\M,\omega}$ be the ``cone-like" region
with the vertex at the origin and the base $\nu^{-1}(\omega) \subset \M$, namely
\beqs
\tC:=\{z\in \R^{n+1}:\; z = \lambda \nu^{-1}(x),\;\lambda\in [0,1],x\in \omega\}.
\eeqs
It is well-known that the volume element of $\tC$ can be expressed by
\beq\label{vol formula}
d\vol(\tC) =\frac{1}{n+1} \frac{u(x)}{K(p)} dx = \frac{1}{n+1} r^{n+1}(\xi) d\xi,
\eeq 
where $p = \nu^{-1}(x) \in \M$,
and $\xi$ and $x$ are associated by
\beq\label{s2 t8}
r(\xi)\xi = u(x) x + \nabla u(x),
\eeq
namely $p=\nu^{-1}(x) = \vec{\hskip1pt r}(\xi)$.
By the second equality in \eqref{vol formula}, we find that
the determinant of the Jacobian of the mapping $x \mapsto \xi = \A_\M^*(x)$ is given by
\beq\label{s2 t9}
|\Jac \A^* |= \left| \frac{d\xi}{d x} \right| = \frac{u(x)}{ r^{n+1}(\xi)K(p)}.
\eeq

\begin{lemma}\label{descending flow}
The functional \eqref{functional} is non-increasing along the normalised flow \eqref{normalised flow}.
Namely $\frac{d}{dt} \J_\alpha(\M_t) \le 0$,
and the equality holds if and only if $\M_t$ satisfies the elliptic equation \eqref{soliton sol}.
\end{lemma}

\begin{proof}
For $\alpha \not = n+1$, it is easy to see
\beq\label{s2 lem1 t1}
 \frac{d}{dt} \J_\alpha(\M_t) = \int_{\S^n} f(x)\frac{u_t}{u} dx -\int_{\S^n} \frac {r_t}{r^{1-q}} d\xi.
\eeq
Let $x=x(\xi,t)=\nu(r(\xi,t)\xi)$. By \eqref{s2 t8} we have
$$\log r(\xi,t) = \log u(x,t) -\log (x\cdot\xi).$$
Differentiate the above identity and denote $\dot{x} = \p_t x(\xi,t)$.
We obtain
\beqn\label{s2 lem1 t2}
\frac{r_t}{r}(\xi,t) &=& \frac{u_t + \nabla u \cdot \dot{x}}{u} - \frac{\dot{x} \cdot \xi}{x \cdot \xi} \\
                         &=& \frac{u_t +( \nabla u-r\xi) \cdot \dot{x} }{u} \notag \\
                          &=& \frac{u_t}{u} (x,t) \notag.
\eeqn
This identity can be also seen from \eqref{normalised flow spt}, \eqref{def v} and \eqref{normalised flow rad}.
Plugging \eqref{s2 lem1 t2} in \eqref{s2 lem1 t1} and then using \eqref{s2 t9} to change the variables,
we obtain
\beqs
\frac{d}{dt} \J_\alpha(\M_t) &=& \int_{\S^n} \frac{u_t}{u}\big(f- \frac{u}{r^{n+1-q} K}\big) dx \\
                                                  &=&- \int_{\S^n} \frac{\big(f r^\alpha K- u\big)^2}{ur^\alpha K} dx \\
                                                  &\le& 0.
\eeqs
Clearly $\frac{d}{dt} \J_\alpha (\M_t) =0$ if and only if $$f(x) r^\alpha(\xi,t) K(p)-u(x,t) =0.$$
Namely $\M_t$ satisfies \eqref{soliton sol}.

When $\alpha = n+1$, we have by differentiating \eqref{functional}
\beqs
\frac{d}{dt} \J_{n+1}(\M_t) = \int_{\S^n} f(x) \frac{u_t}{u} dx - \int_{\S^n} \frac{r_t}{r} d\xi.
\eeqs
By \eqref{s2 t9} and \eqref{s2 lem1 t2}  we get
\beqs
\frac{d}{dt} \J_{n+1}(\M_t) &=& \int_{\S^n}  \frac{u_t}{u} \big( f(x) -  \Big | \frac{d\xi}{dx} \Big | \big )dx \\
                                &=& -\int_{\S^n}  \frac{( f r^{n+1}K -  u)^2}{ur^{n+1}K} dx \\
                                &\le& 0.
\eeqs
This completes the proof.
\end{proof}

The next lemma shows that the functional $\J_{n+1}$ is also monotone along the Gauss curvature flow
for origin-symmetric solutions.
This lemma is of interest itself, 
though it is not needed in the proof of our main theorems. 

\begin{lemma}\label{lem2.2}
Let $\M_t$ be a family of smooth, closed, uniformly convex and origin-symmetric hypersurfaces
which evolve under the normalised Gauss curvature flow \eqref{normalised flow}
with $\alpha=0$ and $f\equiv 1$.
Assume $\vol(\M_0)=\vol(B_1)$.
Then $\frac{d}{dt}\J_{n+1}(\M_t)\le 0$,
and the equality holds if and only if $\M_t$ is the unit sphere.
\end{lemma}

\begin{proof}
Let $\Omega_t$ denote the convex body whose boundary is $\M_t$.
Note that the functional $\mathcal J_{n+1}$ is unchanged under dilation.
The volume $\vol(\Omega_t)$ is preserved under the normalised Gauss curvature flow
\begin{equation}\label{0}
\partial_t u = - K +u,
\end{equation}
where $u(\cdot,t)$ is the support function of $\M_t$.
This can be easily seen from the following evolution equation
\beqs
\frac{d}{dt}\vol(\Omega_t) &=& \frac{1}{n+1}\frac{d}{dt} \int_{\S^n}\frac{u}{K}dx \\
&=&\int_{\S^n} \frac{u_t}{K} dx \\
&=&(n+1)\big(-\vol(\Omega_0)+\vol(\Omega_t)\big).
\eeqs
Hence
\begin{equation}\label{1}
\int_{\mathbb S^n} \frac{u}{K} dx =|\mathbb S^n|, \ \ \ \forall \ t\ge0.
\end{equation}
By the H\"older inequality
$$\Big(\int_{\mathbb S^n} dx\Big)^2 \le\Big( \int_{\mathbb S^n} \frac{K}{u}dx\Big)
\Big(\int_{\mathbb S^n} \frac{u}{K}dx\Big).$$
This together with \eqref{1} shows
\begin{equation}\label{2}
\int_{\mathbb S^n}\frac{K}{u}dx \ge \int_{\mathbb S^n} dx, \ \ \ \forall \ t\ge 0 .
\end{equation}
Recall that Blaschke-Santol\'o inequality of origin-symmetric convex body gives
\beq\label{sb inequality}
\vol(\Omega_t)\vol(\Omega^*_t)\le \vol^2(B_1),
\eeq
where $\Omega^*$ is the polar dual of $\Omega$. Therefore
$$\Big(\int_{\mathbb S^n} dx\Big)^2\ge \int_{\mathbb S^n}\frac{u}{K}dx\int_{\mathbb S^n} (r^*)^{n+1}d\xi^*
= \int_{\mathbb S^n}\frac{u}{K}dx\int_{\mathbb S^n} \frac{1}{u^{n+1}}dx.$$
This together with \eqref{1} implies that
\begin{equation}\label{3}
\int_{\mathbb S^n} \frac{1}{r^{n+1}}dx\le \int_{\mathbb S^n} \frac{1}{u^{n+1}}dx \le  \int_{\mathbb S^n} dx.
\end{equation}
Combining \eqref{2} and \eqref{3}, we conclude that, under the flow \eqref{0},
\beqs
\frac{d}{dt} \J_{n+1} &=& \int_{\S^n}  \frac{u_t}{u} dx - \int_{\S^n} \frac{r_t}{r} d\xi \\
&=& -\int_{\S^n}  \frac{K}{u} dx + \int_{\S^n} \frac{K}{u} \Big|\frac{d\xi}{dx}\Big| dx \\
&=&-\int_{\mathbb S^n} \frac{K}{u} dx + \int_{\mathbb S^n} \frac{1}{r^{n+1}} dx\\
&\le&0.
\eeqs
The last equality holds if and only if the equality in \eqref{2} and in \eqref{3} holds, by \eqref{sb inequality}
it occurs  when $\M_t=\S^n$ only.
\end{proof}

\section{A priori estimates I}

In this section we establish the uniform positive upper and lower bounds
for the support function of the normalised flow \eqref{normalised flow}.

\begin{lemma}\label{s3 lem1}
Let $u(\cdot,t)$, $t\in (0,T]$,  be a smooth, uniformly convex solution to \eqref{normalised flow spt}.
If $\alpha > n+1$, then there is a positive constant $C$ depending only on $\alpha$,
and the lower and upper bounds of $f$ and $u(\cdot,0)$ such that
\beq\label{s3 lem1 1}
1/C \le u(\cdot,t)\le C\ \ \  \forall\ t\in (0,T].
\eeq
If $\alpha = n+1$ and $f \equiv 1$, then
\beq\label{s3 lem1 2}
\min_{\S^n} u(\cdot,0) \le u(\cdot,t) \le \max_{\S^n} u(\cdot,0)\ \ \  \forall\ t\in (0,T].
\eeq
\end{lemma}

\begin{proof}
Let $u_{\min}(t) = \min_{x\in\S^n} u(x,t)$.
By \eqref{normalised flow spt} we have
\beq\label{s3 lem1 t1}
\frac{d}{dt} u_{\min} \ge -(f u^{-q}_{\min}-1)u_{\min}.
\eeq
where $q = n+1-\alpha \le 0$.

If $\alpha >n+1$, we may assume that $ u_{\min}(t) < (\max_{\S^n} f)^{\frac1q} $,
otherwise we are through. Hence $\frac{d}{dt} u_{\min} \ge 0$.
This implies
$$u(\cdot,t) \ge \min \big\{(\max_{\S^n} f)^{\frac1q},\min_{\S^n} u(\cdot,0)\big\}.$$
 Similarly we have
$$u(\cdot,t) \le \max\big\{(\min_{\S^n}f)^{\frac1q}, \max_{\S^n} u(\cdot,0)\big\}.$$
This proves \eqref{s3 lem1 1}

If $\alpha = n+1$ and $f \equiv 1$, then \eqref{s3 lem1 t1} gives $\frac{d}{dt} u_{\min} \ge 0$.
Similarly we have $\frac{d}{dt} u_{\max} \le 0$.
Therefore \eqref{s3 lem1 2} follows.
\end{proof}

When $\alpha = n+1$, for general positive function $f$ which
satisfies \eqref{Aleks f cdt1} and \eqref{Aleks f cdt2},
we can use a barrier argument to derive the $L^\infty$-norm estimates.

\begin{lemma}\label{s3 lem1b}
Let $u$ be as in Lemma \ref{s3 lem1}. 
If $\alpha = n+1$, and $f$ satisfies \eqref{Aleks f cdt1} and \eqref{Aleks f cdt2},
then there is a positive constant $C$ depending only on $\min_{\S^n}u(\cdot,0)$, $\max_{\S^n}u(\cdot,0)$
and $f$ such that
\beq\label{s3 lem1b est}
1/C\le u(\cdot,t) \le C \ \ \forall \ t\in (0,T].
\eeq
\end{lemma}

Before proving Lemma \ref{s3 lem1b}, 
we recall the existence of generalised solutions to Aleksandrov's problem,
of which the proof consists of two steps. The first one is to prove
the existence of polyhedron $\N^*_k$ whose integral Gauss curvature
is a discrete measure converging weakly to $f$. 
Noting that  the integral Gauss curvature is invariant under dilation, 
one may assume that the diameter of $\N^*_k$ is equal to 1. 
Hence by convexity $\N^*_k$ converges to a limit $\N^*$. 
In the second step one uses condition  \eqref{Aleks f cdt2} to show 
that $\N^*$ contains nonempty interior and the origin is an interior point. 
The proof of the second step is also elementary, see page 520, lines 17-27,  \cite{Pog73}.

\begin{proof} [Proof of Lemma \ref{s3 lem1b}]
Let $\N$ be the polar dual of the generalised solution $\N^*$, defined in \eqref{polar dual}.
We use $\N$ as barrier to prove \eqref{s3 lem1b est}.
Let $\M_t$ be a smooth convex solution to the normalised flow \eqref{normalised flow}.
Let $\N_0 = s_0 \N$ and $\N_1 = s_1 \N$, where the constants  $s_1>s_0>0$
are chosen such that $\N_0$ is strictly contained in $\M_0$ and $\M_0$ is strictly contained in $\N_1$.
Let $r_t, \rho_0, \rho_1$ be respectively the radial functions of $\M_t, \N_0, \N_1$.
Note that for any constant $s>0$, 
$s\N$ is a stationary solution to \eqref{normalised flow} in the generalised sense.

We claim that $\M_t$ is contained in $\N_1$ for all $t>0$. 
For if not, there exists a time $t_0>0$ such that  $\sup_{\xi\in \S^n} r_{t_0}(\xi)/\rho_1(\xi)=1$.  
Denote $G=\M_{t_0}\cap\N_1$ ($G$ can be a point or a closed set). 
Since $\frac{\p}{\p t} r_t(\xi)$ is smooth in both $\xi$ and $t$, 
replacing $\rho_1$ by $(1+a)\rho_1$ for a small constant $a$, 
we may assume that the velocity of $\M_t$ is positive on $G\times \{t_0\}$, 
and so also in a neighbourhood of $G\times \{t_0\}$. 
Therefore there exist sufficiently small constants $\eps, \delta>0$, 
such that the velocity of $\M_t$ is greater than $\delta$ at $\xi r_{t_0}(\xi)\in \M_{t_0}$, for all $\xi\in\omega$, 
where $\omega=\{\xi \in\S^n:\ r_{t_0}(\xi)>(1-\eps) \rho_1(\xi)\}$.
By equation \eqref{normalised flow rad}, it means the Gauss curvature of $\M_{t_0}$ 
is strictly smaller than that of $\N_1$  for all $\xi\in\omega$. 
Applying the comparison principle for generalised solutions of the elliptic Monge-Amp\`ere equation
to the functions $r_{t_0}$ and $(1-\eps) \rho_1$,  we reach a contradiction. 

Similarly we can prove that $\N_0$ is contained in $\M_t$  for all $t>0$.
\end{proof}

For $\alpha <n+1$,
we consider the origin-symmetric hypersurfaces
and give the following $L^\infty$-norm estimates.

\begin{lemma}\label{s3 lem1c}
Let $\M_t$, where $t\in (0,T]$, be an origin-symmetric, uniformly convex solution
to the normalised flow \eqref{normalised flow},
and $u(\cdot,t)$ be its support function.
For $ \alpha<n+1$,
there is a positive constant $C$ depending only on $\alpha$, $\M_0$ and $f$, such that
\beq\label{s3 lem1c est}
1/C\le u(\cdot,t) \le C \ \ \forall \ t\in (0,T].
\eeq
\end{lemma}

\begin{proof}

Let us denote by $\mathcal I_q(\M_t)$ the $L^q$ integral of the radial function $r(\xi,t)$, i.e.,
\beqs
\I_q(\M_t)= \int_{\S^n} r^q (\xi,t)d\xi,
\eeqs
where $q=n+1-\alpha$.
By \eqref{normalised flow rad}, we have
\beqs
\frac{d}{dt} \I_q (\M_t)
&=& q \int_{\S^n} r^{q-1} \Big(-\frac{r}{u} f(x) r^\alpha K +r \Big) d\xi \\
&=& -q \int_{\S^n} f(x) \frac{r^{n+1} K}{u} d\xi +q\int_{\S^n} r^q d\xi,
\eeqs
where $f$ takes value at $x=\nu(\xi,t)$ given by \eqref{s2 t5}.
By the variable change formula \eqref{s2 t9}, we obtain
\beqs
\frac{d}{dt} \I_q(\M_t)= q \Big( -\int_{\S^n} f(x) dx + \I_q(\M_t) \Big).
\eeqs
Solving this ODE, one sees
\beq
\I_q(\M_t)=  e^{qt} \Big(\I_q(\M_0) - \int_{\S^n} f \Big) +\int_{\S^n} f
\eeq
It follows that, by our choice of the rescaling factor $\phi_0$ in \eqref{thmDa phi0},
\beq\label{sym Lq-r}
\I_q(\M_t) \equiv \int_{\S^n}f(x)dx, \ \ \forall \ t\in(0,T].
\eeq

Let $r_{\min}(t) =\min_{\S^n} r(\cdot,t)$ and $r_{\max}(t)=\max_{\S^n} r(\cdot,t)$.
By a rotation of coordinates we may assume that $r_{\max}(t) = r(e_1,t)$.
Since $\Omega_t$ is origin-symmetric,
the points $\pm r_{\max}(t) e_1\in \M_t$.
Hence
$$u(x,t) =\sup\{p\cdot x: \ p\in \M_t\}\ge r_{\max}(t) |x\cdot e_1|, \ \forall \ x\in \S^n.$$
Therefore
\beqn\label{sym lem3 t1}
\int_{\S^n} f(x) \log u(x,t)  dx
&\ge&\Big( \int_{\S^n} f(x)dx\Big) \log r_{\max}(t) +\int_{\S^n} f(x) \log |x\cdot e_1| dx \notag\\
&\ge& |\S^n|(\min_{\S^n} f ) \log r_{\max}(t) - C \max_{\S^n} f.
\eeqn
By Lemma \ref{descending flow} and \eqref{sym Lq-r}, we conclude
\beqs
\J_{\alpha}(\M_0)\ge \J_{\alpha}(\M_t)
=\int_{\S^n}f(x)\log u(x,t) dx -\frac1q \int_{\S^n} f.
\eeqs
This together with \eqref{sym lem3 t1} implies
\beq\label{sym ubd}
r_{\max}(t) \le C_1 e^{C_2 \J_\alpha(\M_0)} \le C.
\eeq
This proves the upper bound in \eqref{s3 lem1c est}.

Next we derive a positive lower bound for $u(\cdot,t)$.
We divide it into two cases.

Case (i), $q \in(0,n+1]$. By H\"older inequality,
$$\mathcal I_q (\M_t) \le \mathcal I_{n+1}^{\frac{q}{n+1}}(\M_t) |\S^n|^{\frac{\alpha}{n+1}}. $$
Hence
\beq\label{sym lem2 t1}
\frac{|\S^n|^{-\frac{\alpha}{q}}}{n+1}\mathcal I_q^{\frac{n+1}{q}}(\M_t)
\le \frac{1}{n+1} \mathcal I_{n+1}(\M_t) 
= \vol(\Omega_t),
\eeq
where $\Omega_t$ denotes the convex body enclosed by $\M_t$.
Assume by a rotation if necessary $r(e_{n+1},t) = r_{\min}(t)$.
Since $\Omega_t$ is origin-symmetric, we find that $\Omega_t$ is contained in a cube
$$Q_t=\{z\in \R^{n+1}: \ 
 -r_{\max}(t) \le z_i\le r_{\max}(t) \ \text{for} \ 1\le i \le n,
 \ -r_{\min}(t)\le z_{n+1}\le r_{\min}(t) \}.$$
Therefore by \eqref{sym lem2 t1}
\beqs
\frac{|\S^n|^{-\frac{\alpha}{q}}}{n+1}\mathcal I_q^{\frac{n+1}{q}}(\M_t)
\le 2^{n+1}r^n_{\max}(t) r_{\min}(t).
\eeqs
By \eqref{sym Lq-r}, the left hand side of the above inequality is a positive constant.
Using \eqref{sym ubd}, we get $\min_{\S^n} u (\cdot,t) = r_{\min}(t) \ge 1/C$.

Case (ii), $q> n+1$. We have
\beqs
\I_q(\M_t) &=& r_{\max}^q(t) \int_{\S^n} \Big(\frac{r(\xi,t)}{r_{\max}(t)}\Big)^q d\xi\\
&\le&  r_{\max}^q(t) \int_{\S^n} \Big(\frac{r(\xi,t)}{r_{\max}(t)}\Big)^{n+1} d\xi \\
&=&(n+1) r_{\max}^{q-n-1}(t)\vol(\Omega_t) \\
&\le& C r^{q-1}_{\max}(t) r_{\min}(t).
\eeqs
The lower bound of $r_{\min}(t)$ now follows from \eqref{sym Lq-r} and \eqref{sym ubd}.
\end{proof}

For convex hypersurface, the gradient estimate is a direct consequence of the $L^\infty$-norm estimate. 

\begin{lemma}\label{s3 lem2}
Let $u(\cdot,t)$, $t\in(0,T]$, be a smooth, uniformly convex solution to \eqref{normalised flow spt}.
Then we have the gradient estimate
\beq\label {3.8} 
|\nabla u (\cdot,t)| \le \max_{\S^n\times (0,T]} u, \;\;\forall \; t\in (0,T].
\eeq
\end{lemma}

\begin{proof}
This is due to convexity.
\end{proof}

Similarly we have the estimates for the radial function $r$. 

\begin{lemma}\label{s3 lem3}
Let $X(\cdot,t)$, $t\in (0,T]$, be a uniformly  convex solution to \eqref{normalised flow}.
Let $u$ and $r$ be its support function and radial function, respectively.
Then
\beq\label {3.9} 
\min_{\S^n\times(0,T]} u \le r(\cdot,t)\le\max_{\S^n\times(0,T]} u\ \ \forall \ t\in (0,T], 
\eeq
and
\beq\label {3.10} 
|\nabla r(\cdot,t)| \le C \ \ \forall \ t\in(0,T],
\eeq
where  $C>0$  depends only on
$\min_{\S^n\times(0,T]} u$ and $ \max_{\S^n\times(0,T]} u$ 
\end{lemma}

\begin{proof} Estimates \eqref{3.9} follow from \eqref{s3 lem1b est} as one has
$\min_{\S^n} u(\cdot,t) = \min_{\S^n} r(\cdot,t)$ and $\max_{\S^n} u(\cdot,t) = \max_{\S^n} r(\cdot,t)$.
Estimate \eqref{3.10} follows from \eqref{3.9} because  by \eqref{def v} we have
$|\nabla r| \le \frac{r^2}{u}.$
\end{proof}

\section{A priori estimates II}

In this section we establish uniform positive upper and lower bounds for the principal curvatures
for the normalised flow \eqref{normalised flow}.
We point out that the curvature estimates in this section are for any $\alpha\in\R^1$.

We first derive an upper bound for the Gauss curvature $K(\cdot,t)$.
\begin{lemma}\label{s3 lem4}
Let $X(\cdot, t)$ be a uniformly convex solution to the normalised flow
\eqref{normalised flow} which encloses the origin for $t\in (0,T]$.
Then there is a positive constant $C$ depending only on
$\alpha$, $ f $, $\min_{\S^n\times(0,T]} u$ and $ \max_{\S^n\times(0,T]} u$, such that
\beq\label{ubd-K}
K(\cdot,t) \le C,\;\;\forall \; t\in (0,T].
\eeq
\end{lemma}

\begin{proof}
Consider the auxiliary function
$$Q = \frac{-u_t}{u - \eps_0} = \frac{fr^\alpha K - u}{u-\eps_0},$$
where $$\eps_0 = \frac12 \min_{x\in \S^n,\;t\in (0,T]} u (x,t) >0.$$
At the point where  $Q$ attains its spatial maximum, we have
\beq\label{s3 lem4 t1}
 0 = \nabla_i Q = \frac{-u_{ti}}{u-\eps_0} +\frac{u_t u_i}{(u-\eps_0)^2},
\eeq
and
\beqn\label{s3 lem4 t2}
 0 \ge \nabla^2_{ij} Q &=& \frac{-u_{tij}}{u-\eps_0} +\frac{u_{ti} u_j +u_{tj} u_i +u_t u_{ij}}{(u-\eps_0)^2}
                                          -\frac{2u_t u_i u_j}{(u-\eps_0)^3}  \notag \\
                                   &=&   \frac{-u_{tij}}{u-\eps_0}     +\frac{u_t u_{ij}}{(u-\eps_0)^2} , 
\eeqn
where \eqref{s3 lem4 t1} was used in the second equality above.
The first inequality in \eqref{s3 lem4 t2} should be understood in the sense of negative-semidefinite matrix.
By \eqref{s3 lem4 t2} and \eqref{principal radii} we infer that
\beq\label{s3 lem4 t3}
-u_{tij}-u_t\delta_{ij} \le (b_{ij} - \eps_0\delta_{ij})Q.
\eeq

Using the equation \eqref{normalised flow spt}, we then have
\beqn\label{s3 lem4 t4}
Q_t &=& \frac{-u_{tt}}{u-\eps_0}  + Q^2  \notag \\
      &=& \frac{fr^\alpha S_n^{-2}}{u-\eps_0}S_n^{ij}(-u_{tij}-u_t\delta_{ij})
             + \frac{\alpha f r^{\alpha-1}}{(u-\eps_0)S_n} r_t + Q + Q^2 \notag \\
      &\le& \frac{fr^\alpha K}{u-\eps_0}(n-\eps_0 H)Q + \frac{\alpha f r^{\alpha-1}K}{u-\eps_0} r_t + Q + Q^2,    
\eeqn
where $H$ denotes the mean curvature of $X(\cdot,t)$.

By \eqref{s2 r} and \eqref{s3 lem4 t1},
\beq\label{s3 lem4 t5}
r_t = \frac{uu_t+\sum u_k u_{kt}}{r} = \frac{\eps_0 u -r^2}{r} Q.
\eeq
Without loss of generality we assume  that $K \approx Q \gg 1$.
Plugging \eqref{s3 lem4 t5}  into \eqref{s3 lem4 t4} and noticing that 
$H \ge n K^{\frac{1}{n}}$, we obtain
\beqs
Q_t \le C_0 Q^2\big(C_1 - \eps_0Q^{\frac1n}\big),
\eeqs
for some $C_0,C_1$ only depending on $\alpha,f$ and the $L^\infty$-norm of $u$.
From the  ode we infer that $Q\le C$ for some $C>0$ depending on $Q(0)$, $C_1$ and $\eps_0$.
Our a priori bound \eqref{ubd-K} follows consequently.
\end{proof}


Next we prove that the principal curvatures of $\M_t$ are bounded by positive constants from both above and below.
To obtain the positive lower bound for the principal curvatures of $\M_t$,
we will study an expanding flow by Gauss curvature for the dual hypersurface of $\M_t$.
This technique was previously used in \cite{BIS,IS13, Ivak15, Ivak16}. 
Expanding flows by Gauss curvature have been studied in \cite{Gerh90, Gerh14, Sch06, Urb90,Urb91}.
Our estimates are also inspired by these works.

\begin{lemma}\label{s3 lem5}
Let $X(\cdot, t)$ be the solution of the normalised flow \eqref{normalised flow} for $t\in (0,T]$.
Then there is a positive constant $C$ depending only on $\alpha$,
$f$, $\min_{\S^n\times(0,T]} u$ and $ \max_{\S^n\times(0,T]} u$, such that
the  principal curvatures of $X(\cdot,t)$ are bounded from above and below
\beq\label{bd-kappa}
C^{-1} \le \kappa_i(\cdot,t) \le C,\;\;\forall \; t\in(0,T], \ \text{and} \ i =1,\ldots,n.
\eeq
\end{lemma}

\begin{proof}
To prove the lower bound in \eqref{bd-kappa},
we employ the dual flow of \eqref{normalised flow},
and establish an upper bound of principal curvature for the dual flow.
This, together with Lemma \ref{s3 lem4},
also implies the upper bound in \eqref{bd-kappa}.

We denote by $\M_t^*$ the polar set of $\M_t = X(\S^n,t)$,
see \eqref{polar dual} for the definition of the polar set.
It is well-known that
\beq\label{s3 e1}
 r(\xi,t)= \frac{1}{u^*(\xi,t)},
\eeq
where $u^*(\cdot,t)$ denotes the support function of $\M_t^*$.
Hence by \eqref{s2 t7}, we obtain the following relation
\beq\label{s3 e2}
\frac{u^{n+2}(x,t)(u^*(\xi,t))^{n+2}}{K(p)K^*(p^*)} = 1,
\eeq
where $p\in \M_t$, $p^*\in \M^*_t$ are the two points satisfying $p\cdot p^* = 1$,
and $x,\xi$ are respectively the unit outer normals of $\M_t$ and $\M^*_t$ at $p$ and $p^*$.
Therefore by equation \eqref{normalised flow rad} we obtain the equation for $u^*$, 
\beq\label{s3 e3}
\p_t u^*(\xi,t) = \frac{(u^*(\xi,t))^{n+3-\alpha}f}{(r^*)^{n+1}K^*} - u^*(\xi,t), \ \xi\in \S^n, \ t\in (0,T],
\eeq
where $$K^* = S_n^{-1}(\nabla^2 u^* + u^* I)(\xi,t)$$
is the Gauss curvature of $\M_t^*$ at the point $p^*=\nabla u^*(\xi,t) + u^*(\xi,t) \xi$, and
$$r^* = |p^*|=\sqrt{|\nabla u^*|^2+(u^*)^2}(\xi,t)$$
is the distance from $p^*$ to the origin.
Note that $f$ takes value at $$x = p^*/|p^*| =\frac{\nabla u^* + u^* \xi}{\sqrt{|\nabla u^*|^2+(u^*)^2}} \in \S^n.$$

By \eqref{s3 e1}, $1/C\le u^* \le C$ and $|\nabla u^*| \le C$
for some $C$ only depending on $\max_{\S^n\times(0,T]}u$, $\min_{\S^n\times(0,T]}u$.

Let $b_{ij}^* = u^*_{ij} +u^* \delta_{ij}$, and $\h^{ij}$ be the inverse matrix of $b^*_{ij}$.
As discussed in Section 2, the eigenvalues of $\b_{ij}$ and $\h^{ij}$ are respectively
the principal radii and principal curvatures of $\M^*_t$.
Consider the function
\beq\label{s3 e4}
w= w (\xi,t,\tau)=\log \h^{\tau\tau} -\beta \log u^* + \frac{A}{2}(\r)^2,
\eeq
where $\tau$ is a unit vector in the tangential space of $\S^n$,
while $\beta$ and $A=A(\beta)$ are large constants to be specified later on.
Assume $w$ attains its maximum at $(\xi_0,t_0)$, along the direction $\tau =e_1$.
By a rotation, we also assume $\h^{ij}$ and $\b_{ij}$ are diagonal at this point.

It is direct to see, at the point where $w$ attains its maximum,
\beqn\label{s3 e5}
0\le \p_t w &=&\b_{11} \p_t \h^{11} -\beta\frac{\u_t}{\u} + A\r \r_t \notag \\
                  &=& -\h^{11}\p_t \b_{11} -\beta\frac{\u_t}{\u} + A \r \r_t,
\eeqn
\beqn\label{s3 e6}
0 &=& \nabla_i w = -\h^{11}\nabla_i\b_{11} -\beta\frac{\u_i}{\u} +A\r \r_{i} \notag \\
   &=& - \h^{11} \u_{i11} - \h^{11}\u_1\delta_{1i}  -\beta\frac{\u_i}{\u} +A\r \r_{i},
\eeqn
where $u^*_{ijk}=\nabla_k u^*_{ij}$ throughout this paper.
We also have 
\beqn\label{s3 e7}
0 \ge \nabla^2_{ij} w&=&-\h^{11}\nabla^2_{ij}\b_{11} +2 \h^{11}  \sum \h^{kk} \nabla_1\b_{ik}\nabla_1\b_{kj} \\
                                  &&  -(\h^{11})^2 \nabla_i \b_{11}\nabla_j \b_{11} 
                             -\beta\frac{\u_{ij}}{\u} +\beta\frac{\u_i\u_j}{(\u)^2}   \notag \\
                                 && +A(\r \r_{ij} + \r_i\r_j),  \notag
\eeqn
where the first inequality means that $\nabla_{ij} w$ is a negative-semidefinite matrix.
Note that $\nabla_k \b_{ij}$ is symmetric in all indices.

The equation \eqref{s3 e3} can be written as
\beq\label{s3 e8}
\log(\u_t+\u) - \log S_n = \log\Big(\frac{(\u)^{n+3-\alpha}}{(r^*)^{n+1}}f\Big)=: \psi(\xi,\u,\nabla \u).
\eeq
Differentiating \eqref{s3 e8} gives
\beqn\label{s3 e9}
\frac{\u_{tk} + \u_k}{\u_t+\u}
&=& \sum \h^{ij}\nabla_k\b_{ij} + \nabla_k\psi  \notag \\ 
&=&\sum  \h^{ij}\u_{kij}+ \sum \h^{ij}\u_j\delta_{ik} + \nabla_k \psi,             
\eeqn
and
\beqn\label{s3 e10}
\frac{\u_{t11} + \u_{11}}{\u_t+\u} - \frac{(\u_{t1}+\u_1)^2}{(\u_t+\u)^2}
&=& \sum \h^{ij}\nabla^2_{11}\b_{ij} - \sum \h^{ii}\h^{jj}(\nabla_1 \b_{ij})^2 + \nabla^2_{11}\psi .    
\eeqn

Dividing \eqref{s3 e5} by $\u_t+\u$ and using \eqref{s3 e10}, we have
\beqn\label{s3 e11}
0 &\le& -\h^{11}\Big(\frac{\u_{11t}+\u_{11}}{\u_t + \u} -\frac{\b_{11}}{\u_t + \u} +1\Big) 
            -\frac{\beta\u_t}{\u(\u_t+\u)} + \frac{A\r \r_t}{\u_t+\u} \notag \\
   &=&  -\h^{11}  \frac{\u_{11t}+\u_{11}}{\u_t + \u}  -\h^{11}+ \frac{1+\beta}{u^*_t+u^*}
            -\frac{\beta}{\u}+\frac{A \r \r_t}{\u_t+\u} \notag \\
   &\le&-\h^{11}\sum \h^{ij}\nabla^2_{11} \b_{ij} + \h^{11}\sum \h^{ii}\h^{jj}(\nabla_1\b_{ij})^2 \\
           && -\h^{11}\nabla^2_{11}\psi 
           + \frac{1+\beta}{u^*_t+u^*} +\frac{A \r \r_t}{\u_t+\u}. \notag
\eeqn
By the Ricci identity, we have
\beqs
\nabla^2_{11}\b_{ij} = \nabla^2_{ij}\b_{11} -\delta_{ij}\b_{11} + \delta_{11}\b_{ij}-\delta_{i1}\b_{1j}+\delta_{1j}\b_{1i}.
\eeqs
Plugging this identity in \eqref{s3 e11} and employing \eqref{s3 e7}, we obtain
\beqn\label{s3 e12}
0&\le& \h^{11}\sum\Big( \h^{11} \h^{ii}(\nabla_i \b_{11})^2 - \h^{ii}\h^{jj}(\nabla_1\b_{ij})^2\Big) 
        +(H^*-n\h^{11})  \notag  \\
      &&  -\beta H^* +C\beta -\beta\sum \h^{ij} \frac{\u_i\u_j}{(\u)^2} -\h^{11}\nabla_{11}^2\psi \notag \\
      && +\frac{1+\beta}{\u_t +\u} +\frac{A\r \r_t}{\u_t+\u} - A\sum\h^{ij}(\r \r_{ij} +\r_i\r_j)  \notag  \\
  &\le& -\beta H^* +C\beta  -\h^{11}\nabla_{11}^2\psi \\
     &&+\frac{1+\beta}{\u_t +\u} +\frac{A\r \r_t}{\u_t+\u} - A\sum\h^{ij}(\r \r_{ij} +\r_i\r_j), \notag 
\eeqn
where $H^* = \sum \h^{ii}$ is the mean curvature of $\M_t^*$.

It is direct to calculate
\beqs
\r_t = \frac{\u\u_t + \sum \u_k\u_{kt}}{\r},
\eeqs
\beq\label{s3 e15}
\r_i = \frac{\u\u_i + \sum \u_k\u_{ki}}{\r} = \frac{\u_i\b_{ii}}{\r},
\eeq
and
\beqs
\r_{ij} = \frac{\u\u_{ij} + \u_i\u_j+ \sum \u_k\u_{kij} +\sum \u_{ki}\u_{kj}}{\r} - \frac{\u_i\u_j\b_{ii}\b_{jj}}{(\r)^3}.
\eeqs
Hence, by \eqref{s3 e9},
\beqs
\frac{\r \r_t}{\u_t+\u} - \sum\h^{ij}(\r \r_{ij} +\r_i\r_j)
&=&\frac{\u\u_t}{\u_t+\u} - \u \sum \h^{ij}\u_{ij}  \notag \\
     &&- \sum\h^{ii}(\u_{ii})^2 -\frac{|\nabla \u|^2}{\u_t+\u} +\sum \u_k \nabla_k \psi . \notag 
\eeqs
Since
\beqs
\frac{\u \u_t}{\u_t+\u} - \frac{|\nabla \u|^2}{\u_t + \u} = \u - \frac{(\r)^2}{\u_t +\u},
\eeqs
and
\beqs
-\u\sum \h^{ij} \u_{ij} - \sum \h^{ii}(\u_{ii})^2
&=&-\u\sum \h^{ii} (\b_{ii} -\u\delta_{ii}) -\sum \h^{ii} (\b_{ii} -u\delta_{ii})^2 \\
&=& n\u-\sum \b_{ii},
\eeqs
we further deduce
\beqn\label{s3 e13}
\frac{\r \r_t}{\u_t+\u} - \sum\h^{ij}(\r \r_{ij} +\r_i\r_j)
&\le & C  - \frac{(\r)^2}{\u_t+\u} + \sum \u_k \nabla_k \psi.
\eeqn
Plugging \eqref{s3 e13} in \eqref{s3 e12}, we get
\beqn\label{s3 e14}
0 &\le&
-\beta H^* +C\beta+CA-\h^{11}\nabla^2_{11}\psi+\frac{1+\beta-A(\r)^2}{\u_t + \u}
 +A \sum\u_k\nabla_k\psi \notag \\
&\le& -\beta H^* +C\beta +CA -\h^{11}\nabla^2_{11}\psi +A\sum\u_k\nabla_k\psi,
\eeqn
provided $A > 2(1+\beta)/\min_{\S^n\times(0,T]}(\r)^2 \ge C(1+\beta)$,
for some $C>0$ only depending on $\max_{\S^n\times(0,T]}u$.

By \eqref{s3 e6} and \eqref{s3 e15}, we have
\beqs
-\h^{11}\nabla^2_{11}\psi + A\sum \u_k\nabla_k \psi
&\le& C \h^{11}( 1+ (\u_{11})^2)+CA \\
       &&-\h^{11} \sum\psi_{\u_k} \u_{k11} + A \sum \psi_{\u_k}\u_k\u_{kk} \\
 &\le& C\h^{11}+ C/\h^{11} + CA+ C\beta.     
\eeqs
Hence \eqref{s3 e14} can be further estimated as
\beqs
0 &\le& -\beta H^* +C\h^{11} +C\beta+CA \\
    &\le& -\frac12\beta \h^{11} +C\beta+CA,
\eeqs
by choosing $\beta$ large.
This inequality tells us the principal curvature of $\M^*$ are bounded from above,
namely
$$\max_{\xi\in \S^n} \kappa_i^*(\xi,t)\le C, \ \ \forall \ t\in (0,T] \ \text{and} \ i=1,\ldots,n.$$
By Lemma \ref{s3 lem4} and \eqref{s3 e2}, we have $K^*(\cdot,t) \ge 1/C$.
Therefore
$$1/C\le  \kappa_i^*(\cdot,t) \le C, \ \ \forall \ t\in (0,T] \ \text{and} \ i=1,\ldots,n.$$
By duality, \eqref{bd-kappa} follows.
\end{proof}

As a consequence of the above a priori estimates, 
one sees that the convexity of the hypersurface $\M_t$ is preserved under the flow \eqref{normalised flow}
and  the solution $X(\cdot, t)$  is uniformly convex.

By  estimates \eqref{bd-kappa},  equation \eqref{normalised flow spt} is uniformly parabolic.
By the $L^\infty$-norm estimates and gradient estimates in Lemmas \ref {s3 lem1}--\ref {s3 lem2},
one obtains the H\"older continuity of $\nabla^2 u$ and $u_t$ by Krylov's theory \cite{Kry87}.
Estimates for higher derivatives then follows from the standard regularity theory of uniformly parabolic equations.
Hence we obtain the long time existence and regularity of solutions 
for the normalised flow \eqref{normalised flow}. 
The uniqueness of smooth solutions to \eqref{normalised flow spt} 
follows from the comparison principle, see Lemma \ref{lem5.1} below.
We obtain the following theorem.

\begin{theorem}\label{long time existence}
Let $\M_0$ be a smooth, closed, uniformly convex hypersurface in $\R^{n+1}$ which encloses the origin.
Let $f$ be a positive smooth function on $\S^n$.
Then the normalised flow \eqref{normalised flow} has a unique smooth,
uniformly convex solution $\M_t$ for all time,
if one of the following is satisfied
\begin{itemize}
\item[(i)]  $\alpha >n+1$;
\item[(ii)] $ \alpha =n+1$, and $f$ satisfies \eqref{Aleks f cdt1}, \eqref{Aleks f cdt2};
\item[(iii)] $\alpha <n+1$, and $\M_t$ is origin-symmetric as long as the flow exists.
\end{itemize}
Moreover we have the a priori estimates
\beq
\|u\|_{C_{x,t}^{k,m}\big(\S^n\times[0,\infty) \big)} \le C_{k,m},
\eeq
where $C_{k,m}>0$ depends only on $k,m,f, \alpha$ and the geometry of $\M_0$.
\end{theorem}


\section{Proofs of Theorems \ref{thmA} - \ref{thmDa}}

In this section we prove the asymptotical convergence of solutions to the normalised flow  \eqref{normalised flow}.
First we prove Theorem \ref{thmA}.

\begin{proof}[Proof of Theorem \ref{thmA}] 
Case i): $\alpha >n+1$.

Let $u(\cdot,t)$ be the solution of \eqref{normalised flow spt}.
By our choice of $\phi_0$ in \eqref{phi_0}, we have
$$a:=\min_{\S^n}u(\cdot,0)\le 1\le \max_{\S^n}u(\cdot,0)=:b.$$
Let us introduce two time-dependent functions
\beqs
{\begin{split}
 u_1 &= [1-(1-a^{q})e^{qt}]^{1/q}, \\
 u_2 &= [1-(1-b^{q})e^{qt}]^{1/q},
\end{split}}
\eeqs
where $q = n+1-\alpha <0$.
It is easy to check that both $u_1$ and $u_2$  satisfy equation  \eqref{normalised flow spt}, 
and the spheres of radii $u_1$ and $u_2$ are solutions of \eqref{normalised flow}.
By the comparison principle, $u_1\le u \le u_2$. Hence
$$(b^{q}-1) e^{qt}\le u^{q}-1 \le (a^{q}-1) e^{qt}.$$
Thus $u$ converges to $1$ exponentially.

To obtain the exponential convergence of $u$ to $1$ in the $C^k$ norms,
we use the following interpolation inequality, see e.g. \cite{Ha82},
\beq\label{interpolation}
\int_{\S^n} |\nabla^k T|^2 \le C_{m,n} \Big(\int_{\S^n}|\nabla^m T|^2\Big)^{\frac{k}{m}}
\Big(\int_{\S^n}|T|^2\Big)^{1-\frac{k}{m}}
\eeq
where $T$ is any smooth tensor field on $\S^n$,
and $k,m$ are any integers such that $0\le k\le m$.
Applying this to $u-1$ and using the fact that all derivatives of $u$ are bounded independently of $t$,
we conclude
$$\int_{\S^n}|\nabla^k u|^2 \le C_{k,\gamma} e^{-\gamma t}$$
for any $\gamma\in(0,\tilde \gamma)$ and any positive integer $k$,
where $\tilde \gamma>0$ is a constant depending only on $q$.
By the Sobolev embedding theorem on $\S^n$, see \cite{Aub82}, we have
\beqs
\|u-1\|_{C^l(\S^n)} \le C_{k,l}\Big(\int_{\S^n} |\nabla^k u|^2+\int_{\S^n}|u-1|^2\Big)^\frac12
\eeqs
for any $k>l+n/2$.
It follows that $\|u(\cdot, t)-1\|_{C^l(S^n)}\to 0$
exponentially as $t\to\infty$ for all integers $l\ge 1$. 
Namely $u(\cdot,t)$ converges to $1$ in $C^\infty$ topology as $t\to \infty$.

Case ii): $\alpha=n+1$. 
We first prove the  following lemma.

\begin{lemma}\label{s4 lem1}
There exist positive constants $C$ and $\gamma$ such that
if $X(\cdot,t)$ is a solution to the normalised flow \eqref{normalised flow},
we have the estimate
 \beq\label{esti1}
\max_{\S^n} \frac{|\nabla r(\cdot,t)|}{r(\cdot,t)} \le C e^{-\gamma t}\ \ \forall \ t>0, 
\eeq
where $r(\cdot,t)$ is the radial function of $X(\cdot,t)$.
\end{lemma}

\begin{proof}
Denote $w = \log r$.  By \eqref{s2 t6} and \eqref{s2 t7},
we have, under a local orthonormal frame,
\beqs
{\begin{split}
g_{ij} &= e^{2w}(\delta_{ij} + w_iw_j),\\
h_{ij} &= e^w(1+|\nabla w|^2)^{-\frac12} (\delta_{ij} + w_iw_j -w_{ij}),
\end{split}}
\eeqs
and
\beq\label{s4 lem1 t0}
K = \frac{\det h_{ij}}{\det g_{ij}} = (1+|\nabla w|^2)^{-\frac{n+2}{2}} e^{-n w} \det a_{ij},
\eeq
where
$$a_{ij} = \delta_{ij} + w_i w_j -w_{ij}.$$
By \eqref{def v}, \eqref{normalised flow rad} and \eqref{s4 lem1 t0},
it is not hard to verify that $w$ satisfies the following PDE
\beq\label{s4 lem1 t1}
 w_t= - (1+|\nabla w|^2)^{-\frac{n+1}{2}} \det a_{ij} + 1.
\eeq

Consider the auxiliary function
$$Q = \frac12 |\nabla w|^2.$$
At the point where $Q$ attains its spatial maximum, we have
$$0=\nabla_k Q = \sum w_i w_{ik},$$
and $\nabla^2_{ij} Q$ is a non-positive matrix
$$0\ge \nabla^2_{ij} Q = \sum w_k w_{kij} + \sum w_{ik}w_{kj}.$$
Denote $\varrho = (1+|\nabla w|^2)^{-\frac{n+1}{2}}$.
By differentiating \eqref{s4 lem1 t1}, we obtain, at the point where $Q$ achieves its spatial maximum,
\beqs
\p_t Q &=& \sum w_k w_{kt} = -\det a_{ij} \sum w_k\varrho_k -\varrho \sum w_k \nabla_k \det a_{ij} \\
         &=& \varrho \sum S_n^{ij}  \nabla_k w_{ij} w_k.
\eeqs
By  the Ricci identity, we have
$$ \nabla_k w_{ij} = \nabla_j w_{ik} + \delta_{ik}w_j - \delta_{ij}w_k.$$
Hence
\beqs
\p_t Q &=& \varrho \sum S_n^{ij} \big( Q_{ij} - w_{ik}w_{kj} + w_iw_j -\delta_{ij} |\nabla w|^2\big) \\
           &\le & \varrho  \big( \max_i S^{ii}_n -\sum S^{ii}_n\big)  |\nabla w|^2.
\eeqs
If $n\ge2$, we get
\beq\label{s4 lem1 s2}
\p_t Q \le -\gamma Q,
\eeq
for some positive constant $\gamma$, where we have used the estimates $\varrho \ge C^{-1}$
and $C^{-1} \le \kappa(\cdot,t) \le C$, which are established in Section 3.
Estimate \eqref{esti1} follows from \eqref{s4 lem1 s2} immediately.

For $n=1$, the equation \eqref{s4 lem1 t1} becomes quasi-linear
\beq\label{s4 lem1 a1}
w_t = \frac{w_{xx}}{1+w^2_x} \ \ \ \text{on} \ \S^1\times[0,\infty).
\eeq 
Let $$\bar w: = \frac{1}{2\pi}\int_{\S^1} w(x,t) dx$$ be the average of $w$.
By the divergence theorem,
\beqs
\frac{d}{dt}\bar w = \frac{1}{2\pi} \int_{\S^1} (\arctan(w_x))_x dx =0.
\eeqs
Hence $\bar w$ is a constant. Then it is simple to calculate
\beqs
\frac{d}{dt} \Big(\frac12 \int_{\S^1} (w-\bar w)^2 \Big)
&=& \int_{\S^1}(w-\bar w)(\arctan w_x)_x dx \\
&=& - \int_{\S^1} w_x \arctan w_x dx.
\eeqs
Note that, $w_x \arctan w_x \ge \delta_0 w_x^2$ for some $\delta_0>0$
depending only on the upper bound of $|w_x|$.
We deduce that, by the Poincar\'e inequality,
\beqs
\frac{d}{dt} \Big(\frac12 \int_{\S^1} (w-\bar w)^2 \Big)
&\le& -\delta_0 \int_{\S^1} w_x^2 dx \le -C\int_{\S^1} (w-\bar w)^2.
\eeqs
This implies $w$ exponentially converges to a constant in $L^2$-norm at $t \to \infty$.
The exponential decay \eqref{esti1} follows
by applying the interpolation inequality \eqref{interpolation} to $w-\bar w$.
\end{proof}

Now we prove Case ii) of Theorem \ref{thmA}. 
Lemma \ref{s4 lem1} implies $|\nabla r(\cdot,t)| \to 0$ exponentially as $t\to \infty$.
 
As in Case i), we infer by the interpolation inequality and the a priori estimates in Section 3,
that $r$ exponentially converges to a constant in the $C^\infty$ topology as $t\to \infty$.
Let us show that the constant must be $1$.

By \eqref{normalised flow rad}, we get
\beqs
\frac{d}{dt}\Big(\int_{\S^n} \log r(\xi,t)d\xi\Big) = \int_{\S^n} \Big(-\frac{r^{n+1}K}{u} +1\Big) d\xi.
\eeqs
By \eqref{s2 t9},
$$\frac{d}{dt}\Big(\int_{\S^n} \log r(\xi,t)d\xi\Big) = 0.$$
Therefore by our choice of $\phi_0$ in \eqref{phi_0}
$$\int_{\S^n} \log r(\xi,t) d\xi = \int_{\S^n} \log r(\xi,0)d\xi = 0.$$
This implies $r(\cdot,t) \to 1$ as $t\to \infty$.
\end{proof}

Recall that the normalised flow \eqref{normalised flow}
is a gradient flow of the functional $\J_\alpha$ (see \eqref{functional} for the definition).
We next complete the proofs of Theorem \ref{thmB} -\ref{thmDa}.

\begin{proof}[Proof of Theorem \ref{thmB}]
By our a priori estimates Lemmas \ref{s3 lem1} and \ref{s3 lem3},
there is a constant $C>0$, independent of $t$, such that
\beq\label{s4 thm t1}
|\J_\alpha (X(\cdot,t))| \le C\;\;\forall \; t\in[0,\infty).
\eeq
By Lemma \ref{descending flow}, we obtain
\beqs
\J_\alpha(X(\cdot,T)) - \J_\alpha(X(\cdot,0)) &=& - \int_0^T\int_{\S^n} \frac{(fr^\alpha K-u)^2}{u r^\alpha K} dxdt\\
                                                                      &\le& -\delta_0 \int_0^T\int_{\S^n} (fr^\alpha K-u)^2 dx dt.
\eeqs
By \eqref{s4 thm t1}, the above inequality implies
there exists a subsequence of times $t_j \to \infty$
such that $\M_{t_j}$ converges to a limiting hypersurface which satisfies \eqref{soliton sol}.

To complete the proof of Theorem \ref{thmB},
it suffices to show that the solution of \eqref{soliton sol} is unique.
Using \eqref{s2 r} and \eqref{s2 Gauss}, the equation \eqref{soliton sol} can be written as
\beq\label{s4 thm t2}
\frac{u}{(u^2+|\nabla u|^2)^{\frac{\alpha}{2}}} \det(\nabla^2 u +u I) = f.
\eeq
Let $u_1$ and $u_2$ be two smooth solutions of \eqref{s4 thm t2}.
Suppose $G = u_1/u_2$ attains its maximum at $x_0\in \S^n$.
Then at $x_0$,
$$0=\nabla \log G = \frac{\nabla u_1}{u_1} - \frac{\nabla u_2}{u_2},$$
and $\nabla^2 \log G$ is a negative-semidefinite matrix at $x_0$
\beqs
0 &\ge& \nabla^2 \log G  \\
   &=& \frac{\nabla^2 u_1}{u_1} -\frac{\nabla u_1\otimes\nabla u_1}{u_1^2} 
                                  - \frac{\nabla^2 u_2}{u_2}+ \frac{\nabla u_2\otimes\nabla u_2}{u_2^2} \\
    &=& \frac{\nabla^2 u_1}{u_1} - \frac{\nabla^2 u_2}{u_2} .                         
\eeqs
By \eqref{s4 thm t2} we get at $x_0$,
\beqs
1 
   &=& \frac{u_2^{n+1-\alpha}}{u_1^{n+1-\alpha}}
           \frac{(1+|\nabla u_1|^2/u_1^2)^{\frac{\alpha}{2}}}{(1+|\nabla u_2|/u_2^2)^{\frac{\alpha}{2}}}
           \frac{\det(u_2^{-1}\nabla^2u_2 +  I)}{\det( u_1^{-1}\nabla^2 u_1 +I)}       \\
    &\ge& G^{\alpha-n-1} (x_0).       
\eeqs
Since $\alpha > n+1$, $G(x_0) =\max_{\S^n} G \le 1$.
Similarly one can show $\min_{\S^n} G \ge 1$.
Therefore $u_1 \equiv u_2$.
\end{proof}

\vskip15pt

\begin{proof}[Proof of Theorem \ref{thmC}]

The long time existence of the flow \eqref{normalised flow} follows from Theorem \ref{long time existence}.
As in the proof of Theorem \ref{thmB}, $\M_t$ converges by a subsequence to a homothetic limit.
To prove the convergence of $\M_t$ along $t\to \infty$,
it suffices to show the limiting hypersurface is unique.

First we claim that if $\M_1$ and $\M_2$ are two smooth solutions to \eqref{soliton sol} for $\alpha =n+1$,
then $\M_1$ and $\M_2$ differ only by a dilation. 
This is a well known result  \cite{Aleks42, Pog73}.
We sketch the proof  in \cite{Pog73}  for reader's convenience.
Assume not, then there is a $\lambda >0$ such that 
\begin{itemize}
\item [(i)] the set
$\omega:=\{\xi\in \S^n: r_{\lambda \M_2}(\xi) \ge r_{\M_1}(\xi)\}$ 
is a proper subset of $\S^n$ with positive measure.
\item [(ii)]  the set $\omega_1 := \A_{\M_1}(\omega)$ is contained in
$\omega_2 := \A_{\lambda \M_2} (\omega)$, and
$ |\omega_2| > |\omega_1|$.
\end{itemize}
But on the other hand, by \eqref{s2 t9}, we have
$${\begin{split}
 & \int_{\omega_1} f = |\A^*_{\M_1}(\omega_1)| =|\omega|, \\
 & \int_{\omega_2} f = |\A^*_{\M_2}(\omega_2)| =  |\A^*_{\lambda \M_2}(\omega_2)|= |\omega|.
\end{split}} $$
Hence $\int_{\omega_1} f = \int_{\omega_2} f $, which is in 
contradiction with (ii) above.

Next we show that 
\beq\label{const int}
\int_{\S^n} \log r (\xi,t)d\xi= \int_{\S^n} \log r (\xi,0)d\xi =\const. 
\eeq
This formula and the above claim imply that $\M_t$ converges to a unique limit.

To prove \eqref{const int}, dividing \eqref{normalised flow rad} by $r$ and 
integrating over $\S^n$,
we obtain, by \eqref{def v},
$$ \frac{d}{dt}\Big( \int_{\S^n} \log r(\xi,t)d\xi\Big)=   -\int_{\S^n} f(x) \frac{r^{n+1}K}{u} d\xi + o_n . $$
By the variable change \eqref{s2 t9} and using \eqref{Aleks f cdt1}, we have
\beqs
\frac{d}{dt}\Big( \int_{\S^n} \log r(\xi,t)d\xi\Big) = -\int_{\S^n} f(x) dx + o_n = 0.
\eeqs
Hence we obtain \eqref{const int}.
\end{proof}

\vskip15pt

\begin{proof}[Proof of Theorem \ref{thmDa}]

Since $f$ is even and $\M_0$ is origin-symmetric,
the solution remains  origin-symmetric for $t>0$.
The long time existence of the flow \eqref{normalised flow} now follows from Theorem \ref{long time existence}.
As in the proof of Theorem \ref{thmB}, $\M_t$ converges by a subsequence to a homothetic limit.
We conclude that $\M_t$ converges in $C^\infty$-topology
to a smooth solution of \eqref{soliton sol} as $t\to \infty$
by using the argument of \cite{And97,GuWg03}.
A tractable proof for this was presented in Section 4 of \cite{GuWg03}.

It remains to show, when $f\equiv1$ and $\alpha \ge 0$,
the only origin-symmetric solitons are spheres.
By \eqref{soliton sol}, a soliton to the flow \eqref{normalised flow spt} satisfies
\beq\label{rf1}
uS_n = (u^2+|\nabla u|^2)^{\frac{\alpha}{2}} \ge u^\alpha.
\eeq
While using \eqref{s3 e2}, the polar body of our soliton satisfies
\beq\label{rf2}
u^*S^*_n =\Big(\frac{((u^*)^2+|\nabla u^*|^2)^{\frac12}}{u^*}\Big)^{n+1} (u^*)^\alpha \ge (u^*)^\alpha.
\eeq
Let us denote by $\Omega$ and $\Omega^*$ the convex bodies whose support functions
are $u$ and $u^*$ respectively.
Integrating \eqref{rf1} and \eqref{rf2} over $\S^n$ and then multiplying yield
\beqn\label{rf3}
\text{Vol}(\Omega)\text{Vol}(\Omega^*)
&\ge& \frac{1}{(n+1)^2}\Big(\int_{\S^n} u^\alpha \Big)\Big(\int_{\S^n} (u^*)^\alpha\Big) \notag \\
&\ge&\frac{1}{(n+1)^2} \Big(\int_{\S^n} (uu^*)^{\frac{\alpha}{2}}\Big)^2.
\eeqn
Note that $uu^* = \frac{1}{rr^*}$, and by definition the polar dual,
$$0<r(\xi)r^*(\xi) = \big\langle r(\xi)\xi, r^*(\xi)\xi\big\rangle \le 1.$$
Hence $uu^* \ge 1$.
It then follows from \eqref{rf3} that
\beq\label{rf4}
\text{Vol}(\Omega)\text{Vol}(\Omega^*) \ge \text{Vol}^2(B_1),
\eeq
where $B_1$ denotes the unit ball in $\R^{n+1}$.
The Blaschke-Stanl\'o inequality tells us
$$\text{Vol}(\Omega)\text{Vol}(\Omega^*) \le \text{Vol}^2(B_1).$$
Therefore by the characterisation of equality cases, $\Omega$ must be an ellipsoid.
By \eqref{rf1} and \eqref{rf2},
we infer that $\Omega=B_1$, otherwise the inequality in \eqref{rf4} would become strict,
which is not possible.

\end{proof}

\section{Proof of Theorem \ref{thmD} }

In this section we show that if $\alpha < n+1$ the flow \eqref{flow} may have unbounded ratio of radii,
namely
\beq\label{5.1}
 \tR (X(\cdot,t)) = \frac{\max_{\S^n} r(\cdot,t)}{\min_{\S^n} r(\cdot,t)} \to \infty\;\;\text{as}\;\;t\to T
\eeq
for some $T>0$.
To prove \eqref{5.1}, we show that $\min_{\S^n} r(\cdot,t)\to 0$ in finite time while 
$\max_{\S^n} r(\cdot,t)$ remains positive.
In contrast, it  is worth mentioning that in \cite{Treib90}, 
the author obtained an a priori bound for the ratio $\max_{\S^n} r / \min_{\S^n} r$
if $r$ is the radial function of the solution to the Aleksandrov problem.

Let $X(\cdot, t)$ be a convex solution to \eqref{flow}.
Then its support function $u$ satisfies the equation
\beq\label{flow-u}
{\left\{
\begin{split}
 \frac{\p u}{\p t}(x,t) &= - f r^\alpha S_n^{-1}(\nabla_{ij}^2 u + u \delta_{ij})(x,t),  \\
 u(\cdot,0) &=u_0.
\end{split}\right.}
\eeq
Given a smooth, closed, uniformly convex hypersurface $\M_0$,
our a priori estimates in Section 3 imply the existence of a smooth,
closed, uniformly convex solution to the flow \eqref{flow} for small $t>0$.
The solution remains smooth until  either the solution shrinks to the origin, 
or \eqref{5.1} occurs at some time $T>0$.

\begin{definition} A time dependent family of convex hypersurfaces $Y(\cdot, t)$ is a {sub-solution} to \eqref{flow-u}
if its support function 
$w$ satisfies 
\beq\label{flow-w}
{\left\{
\begin{split}
 \frac{\p w}{\p t}(x,t) &\ge - f r^\alpha S_n^{-1}(\nabla_{ij}^2 w + w \delta_{ij})(x,t)  , \\
 w(\cdot,0) &\ge u_0,
\end{split}\right.}
\eeq
where $r$ is the radial function of the associated hypersurface. 
\end{definition}

By definition, the hypersurface $\M_0$ (independent of $t$), whose support function is $u_0$,
is a sub-solution to \eqref{flow-u}.
We will use the following comparison principle.

\begin{lemma}\label{lem5.1}
Let $X(\cdot, t)$ be a solution to \eqref{flow} and $Y(\cdot, t)$ a sub-solution.
Suppose $X(\cdot, 0)$ is contained in the interior $Y(\cdot, 0)$.
Then  $X(\cdot, t)$ is contained in the interior $Y(\cdot, t)$ for all $t>0$,  as long as the solutions exist.
\end{lemma}

\begin{proof}
Let $u(\cdot, t)$ and $w(\cdot, t)$ be the support functions of $X(\cdot,t)$ and $Y(\cdot,t)$.
Then $u$ and $w$ satisfy \eqref{flow-u} and \eqref{flow-w} respectively with
$u(x,0)\le w(x,0)$ for all $x\in\S^n$.
For $\lambda>0$, let us denote $u^\lambda(x,t) = \lambda u(x,\lambda^\beta t)$,
where $\beta = \alpha-n-1$.
It is easily seen that $u^\lambda$ solves \eqref{flow-u} with $u^\lambda(\cdot,0)=\lambda u_0$.
Let $\lambda <1$. Then $u^\lambda(\cdot,0)< w(\cdot,0)$.
By the comparison principle for parabolic equation, 
\beq\label{contain t1}
u^\lambda(x,t)  < w(x,t), \forall\ x\in\S^n \ \text{and} \ t>0,
\eeq
as long as the solutions exist.
Sending $\lambda \to 1$, we obtain $u(x,t) \le w(x,t)$.
\end{proof}

Note that in Lemma \ref{lem5.1}, we do not require that $Y(\cdot, t)$
is shrinking. Moreover, it suffices to assume that $Y(\cdot, t)$ is a sub-solution in the viscosity sense.
In particular Lemma \ref{lem5.1}  applies if $Y(\cdot, t)$ is $C^{1,1}$ smooth.

To prove Theorem \ref{thmD}, by the comparison principle (Lemma \ref{5.1}),  
it suffices to construct a sub-solution $Y(\cdot, t)$ 
such that $\min_{\S^n} w(\cdot, t)\to 0$  but $\max_{\S^n} w(\cdot,t)$ remains positive, as $t\to T$
 for some finite time $T>0$. By a translation of time, we show below that there is a sub-solution $Y(\cdot, t)$ 
 for $t\in (-1, 0)$ such that \eqref{5.1} holds as $t\nearrow 0$.

\begin{lemma}\label{lem5.2} For any given a positive function $f$, 
there is a  sub-solution $Y(\cdot, t)$, where $t\in (-1, 0)$, to 
\beq\label{flow-af}
{\left\{
\begin{split}
 \frac{\p u}{\p t}(x,t) &= - af r^\alpha S_n^{-1}(\nabla_{ij}^2 u + u \delta_{ij})(x,t),  \\
 u(\cdot,0) &=u_0.
\end{split}\right.}
\eeq
for a sufficiently large constant $a>0$,
such that $\min_{\S^n} w(\cdot, t)\to 0$  but $\max_{\S^n} w(\cdot,t)$ remains positive, as $t\nearrow 0$.
\end{lemma}

\begin{proof}
The sub-solution we constructed is a family of closed convex hypersurfaces $\widehat \M_t :=Y(\S^n, t)$.
First note that it suffices to prove Lemma \ref{lem5.2} for $q=n+1-\alpha>0$ is small. 
Indeed, if $Y(\S^n, t)$ is a sub-solution to \eqref{flow-af} for some $\alpha$, it is also a sub-solution 
to \eqref{flow-af} for $\alpha'<\alpha$, provided we replace $a$ by 
$a\sup\{|p|^{\alpha-\alpha'};\ \ p\in \widehat \M_t, t\in (-1, 0)\}$.

Near the origin, let $\widehat \M_t$ be the graph of a function on $\R^n$,
$\phi (\rho, t)$  ($\rho=|x|$), given by
\beq\label{sub-solu}
\phi(\rho, t)= \left\{ {  \begin{split}
  & - |t|^{\theta} + |t|^{-\theta+\sigma\theta} \rho^2,\ \ \ \text{if}\ \rho < |t|^{\theta}, \\
  & - |t|^{\theta} - \frac{1-\sigma}{1+\sigma} |t|^{\theta(1+\sigma)} 
        +\frac{2}{1+\sigma}  \rho^{1+\sigma},  \ \  \text{if}\  |t|^{\theta} \le \rho \le 1,
  \end{split}}\right.\eeq
where 
$\sigma=\frac{q\theta-1}{n\theta}$ and $\theta>\frac 1q$ is a constant.  
It is easy to verify that $\phi$ is strictly convex, and $\phi\in C^{1,1}(B_1(0))$. 

By direct computation, we have, 
\begin{itemize}
\item [(i)] if $0\le \rho\le |t|^\theta$, then
\beq
{  \begin{split}
 &r^\alpha K  \ge |t|^{\alpha\theta} |t|^{n\theta(\sigma-1)}=|t|^{\theta-1},\\
  & |\frac{\p}{\p t} Y(p, t) | \le \theta |t|^{\theta-1}.
\end{split}} \eeq
where $p=(x, \phi(|x|, t))$ is a point on the graph of $\phi$ and 
$K$ is the Gauss curvature of the graph of $\phi$ at $p$. 
\item [(ii)] if $|t|^\theta\le \rho\le 1$, then
\beq
{  \begin{split}
 & r^\alpha K  \ge \rho^\alpha K \ge C\rho^\alpha \rho^{(\sigma-1)n} 
   \ge C \rho^{1-\frac 1\theta}\ge C |t|^{\theta-1},\\
  & |\frac{\p}{\p t} Y(p, t) | \le \theta |t|^{\theta-1}.
\end{split}} \eeq
\end{itemize}
Hence the graph of $\phi(\cdot, t)$ is a sub-solution to \eqref{flow-af}, 
provided $a$ is sufficiently large.

Next we extend the graph of $\phi$ to a closed convex hypersurface $\widehat \M_t$, 
such that it is $C^{1,1}$ smooth, uniformly convex, rotationally symmetric, and depends smoothly on $t$.
Moreover we may assume that the ball $B_1(z)$ is contained in the interior of $\widehat \M_t$, 
for all $t\in (-1, 0)$, where $z=(0, \cdots, 0, 10)$ is a point on the $x_{n+1}$-axis. 
Then $\widehat \M_t$ is a sub-solution to \eqref{flow-af},  for sufficiently large $a$.
\end{proof}

We are in position to prove Theorem \ref{thmD}. 
For a given $\tau\in (-1, 0)$, 
let $\M_0$ be a smooth, closed, uniformly convex hypersurface
inside $\widehat \M_{\tau}$ and enclosing the ball $B_1(z)$.
Let $\M_t$ be the solution to the flow \eqref{flow-af} with  initial data $\M_0$.
By Lemma \ref{5.1}, $\M_t$ touches the origin at $t=t_0$,
for some $t_0\in (\tau, 0)$. 
We choose $\tau$ very close to 0,
so that $t_0$ is sufficiently small.

On the other hand, let $\tilde X(\cdot, t)$ be the solution to 
\beq\label{s5 t1}
\frac{\p X}{\p t} = -   \beta \tilde f \tilde r^\alpha K \nu,
\eeq
with initial condition $\tilde X(\cdot, \tau)=\p B_1(z)$,
where 
$\beta=2^\alpha\sup\{ |p|^\alpha:\ p\in \M_t, \tau <t<t_0\}$, 
 $\tilde f =a  \max_{\S^n} f$, and
$\tilde r = |X-z|$ is the distance from $z$ to $X$. 
We can choose $\tau$ close enough to $0$ that the ball $B_{1/2}(z)$ is contained in the interior of
$\tilde X(\cdot, t)$ for  all $t\in (\tau, t_0)$. 
Since $\M_t$ is a sub-solution to \eqref{s5 t1}, 
by the comparison principle,
we see that  the ball $B_{1/2}(z)$ is contained in the interior of
$ \M_t $ for  all $t\in (\tau, t_0)$. Hence as $t\nearrow t_0$, 
we have $\min r(\cdot, t)\to 0$ and $\max r(\cdot, t)>|z|=10$. Hence \eqref{5.1} is proved for $\M_t$.
 
We have proved Theorem \ref{thmD} when $f$ is replaced by $af$, for large constant $a>0$.
Making the rescaling $\widetilde \M_t = a^{-1/q} \M_t$,
one easily verifies that $\widetilde \M_t$ solves the flow \eqref{flow} for the function $f$.
Theorem \ref{thmD} is proved.

\vskip10pt

Finally we point out that if $f$ does not satisfy \eqref{Aleks f cdt2}, 
then \eqref{unbounded ratio} holds for $\alpha = n+1$.
Indeed, assume to the contrary that the ratio $ \frac{\max_{\S^n} r(\cdot,t)}{\min_{\S^n} r(\cdot,t)}$
is uniformly bounded, then by \eqref{const int}, the radial function
$r(\cdot, t)$ is uniformly bounded from both above and below.
Hence by the a priori estimates (Lemmas \ref{s3 lem4} and \ref{s3 lem5}), 
the flow converges smoothly to a limit which solves \eqref{soliton sol}.
It means the Aleksandrov problem has a smooth solution without  condition \eqref{Aleks f cdt2}.
But this is impossible as \eqref{Aleks f cdt2} is necessary for the solvability of the Aleksandrov problem.


\end{document}